\documentclass[12pt]{amsart}
\usepackage{amssymb,epsfig,amsmath,latexsym,amsthm}
\theoremstyle{plain}
\newtheorem{theorem}{Theorem}[section]
\newtheorem{lemma}[theorem]{Lemma}

\newtheorem{proposition}[theorem]{Proposition}
\newtheorem{corollary}[theorem]{Corollary}

\newtheorem{conjecture}[theorem]{Conjecture}

\theoremstyle{definition}
\newtheorem{definition}[theorem]{Definition}
\newtheorem{remark}[theorem]{Remark}

\newtheorem{remarks}[theorem]{Remarks}

\newtheorem{example}[theorem]{Example}

\DeclareMathOperator{\hyp}{hyp}

\DeclareMathOperator{\length}{length}
\DeclareMathOperator{\area}{area}

\DeclareMathOperator{\inte}{int}

\DeclareMathOperator{\injrad}{injrad}

\newcommand{\piinv}{\pi^{-1}}

\newcommand{\BR}{\mathbb R}

\newcommand{\BN}{\mathbb N}

\newcommand{\BZ}{\mathbb Z}

\newcommand{\mB}{\mathcal{B}}

\newcommand{\mD}{\mathcal{D}}

\newcommand{\mI}{\mathcal{I}}

\newcommand{\mL}{\mathcal{L}}

\newcommand{\mP}{\mathcal{P}}

\newcommand{\mV}{\mathcal{V}}

\usepackage{epsfig,color}

\makeatother

\theoremstyle{definition}

\def\fnum{equation} 
\newtheorem{Thm}[\fnum]{Theorem}
\newtheorem{Cor}[\fnum]{Corollary}

\newtheorem{Lem}[\fnum]{Lemma}

\newtheorem{Def}[\fnum]{Definition}

\newtheorem{Pro}[\fnum]{Proposition}

\numberwithin{equation}{section}

\newcommand{\nn}{{\bf{n}}}
\newcommand{\Ric}{{\text{Ric}}}
\newcommand{\An}{{\text{An}}}

\newcommand{\diam}{{\text {diam}}}
\newcommand{\dist}{{\text {dist}}}

\def\RR{{\bold R}}

\newcommand{\e}{{\text {e}}}

\newcommand{\Area}{{\text {Area}}}

\newcommand{\cB}{{\mathcal{B}}}

\newcommand{\cL}{{\mathcal{L}}}

\newcommand{\eqr}[1]{(\ref{#1})}

\begin{document}

\title{Effective Finiteness  of irreducible Heegaard splittings of non Haken 3-manifolds}

\author{Tobias Holck Colding}\address{Department of Mathematics\\ Massachusetts Institute of Technology\\Cambridge, MA 02139-4307}

\author{David Gabai}\address{Department of Mathematics\\Princeton
University\\Princeton, NJ 08544}

\thanks{The first author
was partially supported by NSF Grant DMS   
11040934 and NSF FRG grant DMS 
 0854774 and the second by DMS-1006553 and NSF FRG grant DMS-0854969}
 
 \email{colding@math.mit.edu and gabai@math.princeton.edu}

\begin{abstract}  The main result is a short effective proof of Tao Li's theorem that a closed non Haken hyperbolic 3-manifold $N$ has at most finitely many irreducible Heegaard splittings.  
 \end{abstract}

\maketitle

\setcounter{section}{-1}

\section{Introduction}\label{S0}  The long standing classification problem in the theory of Heegaard splittings of 3-manifolds is to exhibit for each closed 3-manifold a complete list, without duplication, of all its irreducible Heegaard splittings, up to isotopy.  

	Fundamental results in this direction were obtained by Klaus Johannson \cite{Jo} and Tao Li \cite{Li1}, \cite{Li2},  \cite{Li3}.  Johannson showed that for Haken 3-manifolds, up to Dehn twisting along tori, for each genus $g$ there are only finitely many isotopy classes of Heegaard surfaces of genus-$g$ and these classes are constructible.  Waldhausen \cite{Wa} conjectured that there are only finitely many irreducible Heegaard surfaces of a fixed genus and this finiteness statement answers a corrected form of Waldhausen's conjecture for Haken 3-manifolds.   For non Haken manifolds, Li's second paper showed that up to isotopy any closed 3-manifold $N$ has at most finitely many irreducible Heegaard splittings, thereby proving a strengthened Waldhausen conjecture.  His first paper did this for splittings of a fixed genus. Li's third paper shows how given $g\in \BN$ to effectively construct a finite list of  closed surfaces of genus-$g$ such that each surface is a Heegaard surface and each Heegaard surface of genus-$g$, up to isotopy, lies in this list.
	
	To solve the classification problem for a non Haken manifold $M$ it therefore suffices to first give an \emph{effective} upper bound for the genus of an irreducible Heegaard splitting of $M$, second to give an algorithm to decide whether or not two irreducible Heegaard splittings are isotopic and third to give an algorithm to decide if a Heegaard splitting is irreducible.  In this paper we solve the first of these problems.  In the sequel \cite{CGK} with Daniel Ketover we solve the second and third.


\begin{theorem}  \label{main} If $N$ is a closed non Haken hyperbolic 3-manifold, then there exists an effectively computable $G(N)$ such that any irreducible Heegaard splitting of $N$ has genus bounded above by $G(N)$. \end{theorem}

	The proof relies on two technical results.  

\begin{theorem}   \label{negative branched} If $N$ is a complete finite volume  hyperbolic 3-manifold and $\eta>-1$, then there exists finitely many effectively constructible $\eta$-negatively curved branched surfaces $B_1, \cdots, B_n$ such that any index-$\le 1$ closed surface is carried by some $B_i$.\end{theorem}  


By an $\eta$-negatively curved surface $B$ we mean that any smooth disc embedded in $B$ has sectional curvature $<\eta$.   

	In 1972 Joe Plante proved that any leaf of a codimension-1 foliation without holonomy has polynomial growth, Theorem 6.3 \cite{Pl}.  In \S3 we will observe that his argument shows 
	
\begin{theorem} \label{plante} Let $M$ be a Riemannian 3-manifold and $B\subset M$ a branched surface.  There exists an effectively constructible polynomial $p(B)$ such that any leaf of any measured lamination carried by $B$ has growth bounded by $p(B)$.  \end{theorem}

Putting these two results together we obtain

\begin{corollary}  \label{injectivity}  Let $N$ be a complete finite volume hyperbolic 3-manifold, then there exists an effectively computable $r>0$ such that if $S$ is an index-$\le 1$ closed minimal surface, then for every $x\in S, \injrad_x(S)\le r$.\end{corollary}


It is well known that a complete hyperbolic 3-manifold has only finitely many minimal surfaces of uniformly bounded genus.  For index at most one, this follows computably from normal surface theory and Theorem \ref{negative branched}.  Using Theorem \ref{main} we now have an effective version of the main result of \cite{Li2}.

\begin{corollary}  If $N$ is a closed non Haken hyperbolic 3-manifold, then there exists an effectively computable $H(N)$ such that $N$ has at most $H(N)$ isotopy classes of irreducible Heegaard splittings.\end{corollary}

Here is the idea of the proof of Theorem \ref{main}.  By Casson - Gordon \cite{CaGo} any irreducible Heegaard splitting surface $S$ of $N$ is strongly irreducible and as announced by Pitts - Rubinstein \cite{PR} $S$ is either isotopic to an index-$\le 1$ surface or is obtained from the double cover of an index-0 surface by attaching an unknotted tube between the sheets.  Thus by Theorem \ref{negative branched} it suffices to show that any negatively curved branched surface carries only finitely many strongly irreducible Heegaard  surfaces or 1-sided strongly irreducible Heegaard surfaces and that this bound is effectively computable.  Using Theorem \ref{plante} and its Corollary, it follows that any -1/2-negatively curved  branched surface $E$ can be effectively split into a uniform number of negatively curved branched surfaces $B_1, \cdots, B_m$ such that for each $i$, no component of $\partial_h(B_i)$ is either a disc or an annulus.  Furthermore, any closed surface carried by $E$ is carried by some $B_i$.  We will show that  the $B_i$'s can be constructed so that if $B$ is a subbranched surface of $B_i$ that fully carries a surface, then $B$ has the similar property.  Now let $F_1, \cdots, F_q$ be the set of fundamental solutions of $B_i$.  We effectively find a number $p(B_i)$ such that if $S=n_1F_1+\cdots+n_qF_q$, $S$ is a Heegaard surface and for some $i$, $n_i>p(B_i)$, then $S$ is weakly reducible.  The idea is that if $B$ is the subbranched surface that carries all the $F_i$'s with coefficients $>p(B_i)$, then either $B$ is incompressible and hence $N$ is Haken or $\partial_h N(B)$ is compressible.  In the latter case, in the most interesting situation, a compressing disc for $\partial_h N(B)$ decomposes into two monogons that after a small perturbation gives rise to disjoint monogons for $\partial_hN(B_i)$.  These monogons extend to disjoint compressing discs for $S$ one on each side of $S$.  Each such disc consists of two copies of a monogon and a strip that lies in the interstitial bundle of $S$ and which penetrates only a uniformly bounded amount.  It is this bound that enables us to control $p(B_i)$.  \vskip8pt

\noindent\emph{Comparison with Tao Li's work.}  First, Li works without any assumption on  the Riemannian metric of the underlying manifold.   Second, the idea of splitting to eliminate disc regions of $\partial_hN(B_i)$ is central in his papers as is the idea of using incompressible subbranched surfaces to find incompressible surfaces.  Third, he uses the interstitial bundle to find helix-turn-helix bands which are used to build compressing discs.  The main innovation here is the interplay between Theorems \ref{negative branched} and \ref{plante} which effectively bounds the amount of splitting needed to eliminate discs and annuli lying in the horizontal boundary of $N(B)$ and effectively bounds the injectivity radii at each point of a strongly irreducible Heegaard surface.  That in turn controls how deep one must enter the interstitial bundle to find essential annuli, pieces of which are used to build relatively simple compression discs.    Provided $p(B_i)$ is sufficiently large, there is enough depth to construct these compression discs.  \vskip 8pt

This paper is organized as follows.  Basic definitions and some facts are presented in \S1.  Some fundamental results about minimal surfaces in 3-manifolds are established in \S2. Theorem \ref{negative branched} is proved in \S3 and Theorem \ref{plante} is proved in \S4.  In \S5 we show that given an $\eta$-negatively curved branched surface $B$ a uniformly bounded amount (in $\eta$) of splitting is needed to produce a set of \emph{horizontally large} branched surfaces that in the aggregate carry the same surfaces that $B$ does.   The main result is obtained in \S6.    In \S7 we reduce the problem of finding a purely combinatorial proof to a conjecture about branched surfaces.  
Finally, in the appendix (using results announced in \cite{PR}; see also \cite{Ru} and compare with \cite{Ke}) we show that any closed orientable Riemannian 3-manifold $M$ with a strongly irreducible Heegaard splitting, supports a mean convex foliation.



\section{Basic definitions and facts}\label{basic}

\begin{definition} For the basic definitions regarding measured laminations and branched surfaces see \cite{O}. In particular you will find there pictures of local models of a branched surface $B$, definitions of the \emph{branch locus} $b(B)$, \emph{fibered neighborhood} $N(B)$, the \emph{vertical}  and \emph{horizontal}  boundaries $\partial v N(B)$ and $\partial h N(B)$ as well as definitions of \emph{carrying}, \emph{monogon} and \emph{sector}.  We let $\mV(N(B))$ denote the \emph{$I$-fibering} of $N(B)$ and $\pi:N(B)\to B$ the projection contracting fibers to points. We often think of $\pi$ as being defined on all of the ambient manifold $M$.  All measured laminations are assumed to have full support.  All branched surfaces are implicitly assumed to be associated with an immersion in some 3-manifold thereby enabling us to define $N(B)$.  If $B$ is embedded in the 3-manifold $M$, then the \emph{closed complement} $C(N(B))$ is the closure of $M\setminus N(B)$.  If $B$ carries $\lambda$, then a corresponding embedding of $\lambda$ in $N(B)$ gives rise to the \emph{interstitial bundle} $\mI(\lambda)$ well defined up to fiber preserving homeomorphism.  Here $\mI(\lambda)$ is the $I$-bundle obtained by starting with $N(B)\setminus \lambda$, taking the closure with respect to the induced path metric, then restricting to those $I$-fibers with both endpoints in $\lambda$.  We say that $B$ \emph{fully carries} $\lambda$ if it carries $\lambda$ and each sector of fiber of $\mV(N(B))$ hits some leaf.  Given an explicit embedding of $\lambda\subset N(B)$ transverse to $\mV$ we say that a point $x\in \lambda$ is \emph{interior} to $N(B)$ if there exists an interval $J$ of an $I$-fiber such that $J\cap \partial_v N(B)=\emptyset$, $\partial J\subset \lambda$ and $x\in \inte(J)$. \end{definition}  

 \begin{definition}  We say that $D$ is a \emph{compression disc} for the branched surface $B$ if $D$ is the $\pi$-image of an embedded essential compression disc for $\partial_h N(B)$.   \end{definition}
 
 \begin{remark} \label{compressing}  We will view $D$ as both an embedded  disc with boundary in $\partial_h N(B)$ and as an immersed disc with boundary in $B$.  The difference will be clear from context.  Note that if $B$ is a smooth 1-sided surface, then  a compressing disc for $B$ viewed as a branched surface may not  be a compression disc for $B$ viewed as a smooth surface.\end{remark}

\begin{definition} If $B'$ is obtained by splitting $B$, then let $\Phi:\pi_0(\partial_h N(B))\to \pi_0(\partial_h N(B'))$ be the map induced by inclusion.  We say that $B'$ is a \emph{regular splitting} if $\Phi$ is surjective.  If $B$ carries the lamination $\lambda$ and $B'$ also naturally carries $\lambda$, then $B'$ is said to be obtained by $\lambda$-\emph{splitting} $B$.   \end{definition} 

A branched surface chart in a 3-manifold can be given one of two transverse orientations.  Also if $B$ is a branched surface then there is a  homomorphism $f:\pi_1(B)\to \BZ_2$ such that $f[\gamma]=0$ if and only if $\gamma$ is orientation preserving.  Thus we have the well known  

\begin{proposition}  \label{2-fold} Let $B$ be a branched surface embedded in an 3-manifold $M$.  There exists a 2-fold cover $\hat N(B)\to N(B)$ such that the induced $I$-fibering of $\hat N(B)$ is orientable.\end{proposition}

Assuming that a space $X$ and its metric are understood from context, then the notation $N(x,r)$ will denote the metric $r$-ball about $x$, otherwise we will use $N_X(x,r)$.  If $H$ is a rectifiable path connected subset of the Riemannian manifold $X$, then $H$ will be given the induced path metric.  If $X$ is a space, then $|X|$ will denote the number of components of $X$.  If it is just a set, then $|X|$ denotes the number of elements.


\begin{Def}
A branched surface is \emph{negatively curved} (resp. 
\emph{c-negatively curved}) if each point has sectional curvature $<0$ 
(resp. $<c<0$).
\end{Def}

\begin{Thm} 
Let $N$ be a complete finite volume hyperbolic $3$-manifold and 
$\epsilon >0$. Then there exists a constructible set $B_1, \cdots, B_n$ 
of $(-1+\epsilon)$-negatively curved branched surfaces such that if $S$ 
is an embedded minimal surface of index $\leq 1$, then $S$ is carried by 
some $B_i$.
\end{Thm}


\section{Minimal surfaces in hyperbolic $3$-manifolds}
\subsection{Index one minimal surfaces in a ball and in a $3$-manifold}  \label{s:sindex}

We will state the main result of this subsection below as a result about embedded minimal surfaces in a unit ball in Euclidean $3$-space with boundary in the boundary of the ball.   However, with obvious changes this result holds for a ball in a fixed $3$-manifold provided that the radius of the ball is sufficiently small depending on the center of the ball and we will later use this generalization.   In particular, this results holds in both closed and non-compact manifolds as long as the manifold is complete; the only catch is that the radius of the ball where it holds depend on the local geometry and thus in the case of non-compact manifolds may not have a uniform positive lower bound independent of the center.

\vskip2mm
Let $\Gamma$ be a surface in $\RR^3$ possibly with boundary. We will use $\An_r(\gamma)$ to denote the intrinsic tubular neighborhood of radius $r$ about a curve $\gamma$, i.e.,
\begin{align}
\An_r(\gamma) = \{x \in \Gamma| \dist_{\Gamma}(x,\gamma) < r\}\, .
\end{align}
Similarly, we will write $\An_{s,t}(\gamma)$ for the �annulus� $\An_t(\gamma) \setminus \An_s(\gamma)$. Let $B_\epsilon\subset \BR^3$ denote the ball of radius $\epsilon$ centered at the origin.  

\begin{Thm}   \label{t:mainindex}
There exists $\delta$, $c> 0$ such that the following holds:  Given $C>1$, $\mu> 0$, and an integer $k$, there exists $\epsilon> 0$ so that if $\Sigma\subset B_1\subset \RR^3$ is a compact embedded minimal surface with $\partial \Sigma\subset \partial B_1$ and index one and if $B_{\epsilon}\cap \Sigma$ is unstable,  then there exists a simple closed geodesic $\gamma\subset B_{2\epsilon}\cap \Sigma$ of length $\ell$ so that
\begin{enumerate}
\item $\Sigma_{c}\setminus \An_{\delta \ell}(\gamma)$ consists of two graphical annuli of functions with gradient at most one.   \label{e:e1mt} 
\item $\An_{C\ell}(\gamma)$ is $\mu$-$C^k$ close to the corresponding annulus in an rescaled catenoid with neck of length $\ell$.   \label{e:e2mt} 
\end{enumerate}
Here $\Sigma_c$ is the connected component of $B_c\cap \Sigma$ containing $\gamma$.
\end{Thm} 

As a corollary we get that in any $3$-manifold any limit of any sequence of closed embedded index one minimal surfaces is a smooth lamination.  This is the following:

\begin{Cor}   \label{c:lamination}
Let $M^3$ be a complete $3$-manifold and $\Sigma_i\subset M$ a sequence of closed embedded minimal surfaces with index one.  Then a subsequence converges to a smooth minimal lamination $\cL$.  Moreover, at most one leaf of $\cL$ is unstable and if $L$ is an unstable leaf, then it is isolated.  
\end{Cor}

Recall that a leaf $L$ of a lamination $\cL$ of a $3$-manifold $M$ is said to be isolated if for each $x\in L$, there exists an $\epsilon>0$ such that $L$ is the only leaf of $\cL$ that intersect $B_{\epsilon}(x)$.  Note that an isolated leaf can limit into a non-isolated leaf.

\begin{proof}
(of Corollary \ref{c:lamination}).  
A priori such a limit could be smooth away from a single point where the index concentrate however it follows easily from Theorem \ref{t:mainindex} together with a standard removable singularity theorem (cf. page 119 of \cite{CM1}) that this potentially singular point is a removable singularity.
\end{proof}

To prove Theorem \ref{t:mainindex} we will first show a blow up version of it on `the scale' of the index.  In this next result and the lemma that follow 
$0\in\Sigma_i\subset B_{r_i} =B_{r_i} (0)\subset \RR^3$ will be a sequence of compact embedded minimal surfaces with index one satisfying:
\begin{enumerate}
\item $\partial \Sigma_i\subset \partial B_{r_i}$ where $r_i\to \infty$.  \label{e:blowup1}
\item Each intrinsic unit ball $\cB_1(p)$ in each $\Sigma_i$ is stable.  \label{e:blowup2}
\item The intrinsic ball $\cB_2(0)$ in each $\Sigma_i$ is unstable. \label{e:blowup3}
\end{enumerate}

The blow up result is the following:

\begin{Pro}    \label{p:convergence}
If $\Sigma_i\subset B_{r_i} =B_{r_i} (0)\subset \RR^3$ is a sequence of compact embedded index one minimal surfaces satisfying \eqr{e:blowup1}--\eqr{e:blowup3}, 
then a subsequence $\Sigma_j$ converges in the $C^k$ topology on compact subsets of $\RR^3$ to the catenoid for any $k$.
\end{Pro}

We will use a number of times below that since $\Sigma_i$ has index one and the ball $\cB_2(0)$ is unstable any intrinsic ball in $\Sigma_i$ that is disjoint from $\cB_2(0)$ is stable.\footnote{Within the proof of Proposition \ref{p:convergence} we will show a bit more than what is actually needed (namely, that the limit $\Sigma_{\infty}$ is properly embedded); however we believe that the additional estimates add clarity to the argument.}

\begin{proof}
(of Proposition \ref{p:convergence}.)
Since each unit ball $\cB_1(p)$ in each $\Sigma_i$ is stable for any sequence $s_i\to \infty$ with $r_i\geq 2s_i$ the sequence $\cB_{s_i}(0)\subset \Sigma_i$ has uniformly bounded second fundamental form by \cite{S2}, \cite{CM2}.  We claim that for any fixed $R>0$, there is a uniform bound, independent of $i$, for the area of $B_R\cap \cB_{s_i}(0)$.  Together these two properties gives that a subsequence $\cB_{s_i}(0)$ converges in the $C^k$ topology to a properly embedded minimal surface $\Sigma_{\infty}$ that then is easily is seen to have index one.\footnote{The bound for the second fundamental form gives that locally a subsequence converges smoothly but possibly with multiplicity; the uniform area bound (together with the second fundamental form bound) rules out multiplicity.}  By the classification of index one embedded minimal surfaces in $\RR^3$ (see corollary 1 on page 255 in \cite{CgT} or theorem 2 on page 37 of \cite{LR}) it follows that this limit is the catenoid and from this and since the convergence is in the $C^k$ topology the claim easily follows.

To complete the proof we need therefore only show that for fixed $R>0$, there exists some sequence $s_i$ as above so that there is a uniform bound, independent of $i$, for the area of $B_R\cap \cB_{s_i}(0)$.  This is essentially contained in \cite{CM5} (see lemma II.2.1 there or lemma 8.23 in \cite{CM6}), however, for completeness we include a proof that is adapted to the relatively strong assumptions here\footnote{The same proof that we give below show that as long as the annuli $\cB_{sr_i}(0)\setminus \cB_{R_0}(0)$ are stable for some fixed $R_0>0$ and the surfaces has uniformly bounded curvature, then any limit $\Sigma_{\infty}$ is properly embedded.}.  Suppose therefore that for some $R>0$ fixed and any sequence $s_i$ as above
\begin{align} \label{e:noareabound} 
\Area (B_R\cap \cB_{s_i}(0))\to \infty\, ; 
\end{align}
This easily implies that there exists $x_i$, $y_i\in B_R\cap \Sigma_i$ such that
\begin{align}
d_{\Sigma_i}(0,x_i)&\to \infty\, , \label{e:e1proper}\\
d_{\Sigma_i}(0,x_i)/d_{\Sigma_i}(0,y_i)&\to 0\, , \label{e:e2proper}\\
d_{\Sigma_i}(0,y_i)/r_i&\to 0\, .\label{e:e3proper}
\end{align}
This contradict the next lemma; thus proving the proposition.  
\end{proof}

\begin{Lem}    \label{l:convergence}
Let $\Sigma_i\subset B_{r_i} \subset \RR^3$ be a sequence of compact embedded index one minimal surfaces satisfying \eqr{e:blowup1}--\eqr{e:blowup3}.   
Then there are no sequences $x_i$, $y_i\in B_R\cap \Sigma_i$ satisfying \eqr{e:e1proper}--\eqr{e:e3proper}.
\end{Lem}

\begin{proof}
Suppose not; we will obtain a contradiction.  
Set 
\begin{align}
s_i&=d_{\Sigma_j}(0,x_i)\, , \\
t_i&=d_{\Sigma_j}(0,y_i)\, .
\end{align}
Then for $i$ sufficiently large the intrinsic balls
$\cB_{\frac{5s_i}{6}}(x_i)$, $\cB_{\frac{5t_i}{6}}(y_i)$ 
are stable.  In particular, by Schoen's curvature estimate for stable minimal surfaces, \cite{S2}, \cite{CM2}, the balls $\cB_{s_i}(y_i)$ converges to flat balls in a plane (even after rescaling to unit size as $s_i/t_i\to 0$) and  for some constant $C$
\begin{align}  \label{e:apriori}
\sup_{\cB_{\frac{4s_i}{5}}(x_i)}|A|^2\leq C\,s_i^{-2}\, . 
\end{align}
We claim that it follows from this curvature estimate together with the fact that $\cB_{s_i}(x_i)\cap \cB_{s_i}(y_i)=\emptyset$ for $i$ sufficiently large that the rescaled balls $\cB_{\frac{3s_i}{4}}(x_i)$ converges to a flat ball in a plane centered at the origin.  From this and since $0\notin \cB_{\frac{3s_i}{4}}(x_i)$ it follows from the triangle inequality and since intrinsic distances are larger than extrinsic distances that $d_{\Sigma_i}(0,x_i)>s_i$ for $i$ sufficiently large which is the desired contradiction.

It remains to show that a subsequence of the rescaled balls $\cB_{\frac{3s_i}{4}}(x_i)$ converges to a flat ball in a plane centered at the origin.  This however follows from the following four facts:
\begin{itemize}
\item The rescaled balls $\cB_{\frac{3s_i}{4}}(x_i)$ have uniform curvature bounds by \eqr{e:apriori}.
\item The rescaled distance between the centers of the balls $\cB_{s_i}(y_i)$ and $\cB_{\frac{3s_i}{4}}(x_i)$ and $0$ goes to zero as  $x_i$ and $y_i\in B_R$.
\item $\cB_{s_i}(x_i)\cap \cB_{s_i}(y_i)=\emptyset$ for $i$ sufficiently large.
\item A subsequence of the rescaled balls $\cB_{s_i}(y_i)$ converges to flat balls in a plane.   
\end{itemize}
\end{proof}

\vskip2mm
We will also need proposition D.2 of \cite{CM4}.  This proposition will be used to go from the blow-up scale, where the index one minimal surface looks like the catenoid, to almost all the way out to the boundary of the initial ball.
Loosely speaking this proposition from \cite{CM4} asserts that a stable embedded minimal surface with a single interior boundary curve $\gamma$ and an area bound near $\gamma$ is graphical away from its boundary.

\begin{Pro}   \label{p:appendixD}
(Proposition D.2 of \cite{CM4}.)  
Given a constant $C$, there exists $\omega> 1$ so that if $\Gamma \subset B_R$ is a stable embedded minimal surface whose �interior boundary� $\partial \Gamma\setminus \partial B_R$ is a simple closed curve $\gamma\subset B_4$ satisfying
$\Area(\An_2(\gamma)) \leq C$,  then each component of $B_{R/\omega} \cap \Gamma \setminus B_{\omega}$ is a graph with gradient bounded by one.
\end{Pro}

We will need that this proposition has the following extension to the case where the interior boundary is disconnected:

\begin{Pro}  \label{p:appendixDref}
In Proposition \ref{p:appendixD} instead of assuming that the inner boundary is connected it suffices to assume that for each connected component of the inner boundary $\gamma$
\begin{align}
\An_{2R}(\gamma)\cap \partial \Gamma\subset \partial B_R\, .
\end{align}
All other assumptions as well as the conclusion remains the same.
\end{Pro}

\begin{proof}
In fact, this assumption was all that was needed in the proof of proposition D.2 in \cite{CM4}; cf. also with remark D.14 in \cite{CM4}.
\end{proof}

Let $D_r$ be the disk in the plane centered at the origin and of radius $r$; with polar coordinates $(\rho, \theta)$ so that $\rho > 0$ and $\theta\in \RR$.  Suppose also that $u$ and $v$ are functions on the annulus $D_R\setminus D_r$.  The separation is the function given by
\begin{align}
w(\rho, \theta) = v(\rho, \theta) - u(\rho, \theta) \, .
\end{align}

The argument given on page 50 in the proof of proposition II.2.12 of \cite{CM3} gives that the separation between two minimal graphs defined over an annulus grows sub-linearly:

\begin{Lem}   \label{l:sublinear}
Given $\alpha>0$,  there exists $\delta>0$ and $k>0$ so that the following holds:

If $u$ and $v$ satisfies the minimal surface equation\footnote{So that the graphs of $u$ and $v$ are minimal surfaces.} on $D_{\e^kR}\setminus D_{\e^{-k} R}$, have gradients bounded by $1/2$, and the separation $w$ satisfies 
$0< w< \delta R$, then for $R< s< 2R$
\begin{align}
\sup_{D_s\setminus D_R} w\leq  \left(\frac{s}{R}\right)^{\alpha}\sup_{\partial D_R} w\, .
\end{align}
\end{Lem}

We will also need the following elementary lemma:

\begin{Lem}   \label{l:connected}
Given $C$, $\delta>1$, there exists $\epsilon>0$ such if $\Sigma\subset B_1\subset\RR^3$ is a compact embedded minimal surface with 
\begin{itemize}
\item $|A|^2\leq C\,r^{-2}$ and $\partial \Sigma\subset \partial B_1\cup B_{2r}$.   
\item The connected component of $B_{2r}\cap \Sigma$ containing $\partial \Sigma\setminus \partial B_1$ consists of two graphs over an annulus in a plane through the origin of functions with gradients bounded by $1$. 
\end{itemize}
Then the following holds:  If the separation between these two graphs is at most $\epsilon\,r$, then the intrinsic distance between the two inner boundary components of $\Sigma$ is bounded below by $\delta\,r$. 
\end{Lem}

Lemma \ref{l:sublinear} together with Lemma \ref{l:connected} will be used in the proof of the main theorem of this section to show that the assumption in 
Proposition \ref{p:appendixDref} that the intrinsic distance between inner boundary components is large and thus the proposition applies.

We can now combine the above results to prove the main theorem of this section:

\begin{proof}
(of Theorem \ref{t:mainindex}.)
Suppose not;  then there exists a sequence of compact embedded minimal surfaces $\Sigma_i\subset B_1$ of index one with $\partial \Sigma_i\subset \partial B_1$ and so that $B_{\epsilon_i}\cap \Sigma_i$ is unstable where $\epsilon_i\to 0$.   Moreover, none of the $\Sigma_i$'s has a decomposition as in \eqr{e:e1mt} and \eqr{e:e2mt}.  For each $i$ let $s_i$ be the supremum of all $r>0$ such that the intrinsic balls $\cB_r(p)$ are stable for all $p\in B_{\frac{1}{2}}\cap \Sigma_i$.  It follows that the intrinsic balls $\cB_{s_i}(p)$ are stable for all $p\in B_{\frac{1}{2}}\cap \Sigma_i$ and that $\cB_{2s_i}(p_i)$ is unstable for some $p_i\in B_{\frac{1}{2}}\cap \Sigma_i$.  After translation we may assume that each $p_i$ is the origin in $\RR^3$ and after rescaling by $1/s_i$ we get a sequence of embedded minimal surfaces as in Proposition \ref{p:convergence}.  Claim \eqr{e:e2mt} in the statement of the theorem now follows from Proposition \ref{p:convergence}.   From this \eqr{e:e1mt} follows from Proposition \ref{p:appendixDref} applied to the sequence before translation and rescaling combined with Lemmas \ref{l:sublinear} and \ref{l:connected}.  (Where Lemmas \ref{l:sublinear} and \ref{l:connected} are used to show that Proposition 
\ref{p:appendixDref} applies.)  This is the desired contradiction.  
\end{proof}

\subsection{Index one embedded surfaces in cusps}

Recall that for any finite volume hyperbolic $3$-manifold,  $N^3$,  there exists a compact set $K\subset N$ such that $N\setminus K$ is the union of finitely many ends each of which is of the form $E_k=\Gamma_k\times_{\e^{-r}} (0, \infty)$; that is, a warped product of a half-line with a flat torus $\Gamma_k$, where the wrapping function is $\e^{-r}$.  
(That is, the metric on $E_k$ is $dr^2+\e^{-2r}\,g$, where $g$ is the flat metric on $\Gamma_k$.)  The function $r_k=r$ is the Busemann function to the end.  (Note that with this sign convention of the Busemann function horoballs are super-level sets.)   

For $p\in N$ let $\text{inj}_p$ denote the injectivity radius at $p$ of $N$.    We will use below that it follows, that for each end of $M$, there is a $R_k$ (depending only on the geometry of the flat torus 
$\Gamma_k$), so that the function
\begin{align} \label{e:inj}
p\to \e^{r_k(p)}\, \text{inj}_p
\end{align}
is constant on the horoball $r_k\geq R_k$.  Similarly, on the end $E_k$, there exists a constant $C_k$ so that the entire horosphere $r_k=r_k(p)$ is contained in the ball $B_{C_k\, \text{inj}_p}(p)$.  

In this subsection we will show that in a finite volume hyperbolic $3$-manifold closed embedded index one minimal surfaces does not penetrate deeply into the cusps.    This is the following:

\begin{Thm}
Let $N^3$ be a finite volume hyperbolic $3$-manifold and $x\in N$ a fixed point, then there exists an $R>0$ so that any closed embedded index one minimal surface is contained in the ball $B_R(x)$. \footnote{The same result should hold for a finite volume $3$-manifold with pinched negative sectional curvature.}
\end{Thm}

\begin{proof}
We will show the theorem by showing that if a closed embedded index one minimal surface penetrated deeply into a cusp, then the entire minimal surface would have to be contained in the cusp contradicting the maximum principle.

Let $K$, $E_k$, and $r_k$ be as above and suppose that $\Sigma$ is a closed embedded index one minimal surface in $M$ that intersect $E_k$ and let $p\in E_k\cap \Sigma$ be a point where the maximum of $r_k$ restricted to $\Sigma$ achieved.  (Note that since the horoballs $r_k\geq r$ are strictly convex it follows that any closed minimal surface cannot be entirely contained in $E_k$.)  For $q\in \Sigma$ let  $s_q$ be the radius of the largest intrinsic ball $\cB_s(q)$ that is stable.  Fix $\epsilon>0$ small to be determined (depending only on $C_k$ and the constant in \eqr{e:inj}).  We divide the argument into three cases:
\begin{enumerate}
\item $s_p\geq \epsilon^{-1}\, \text{inj}_p$.\label{e:case1}
\item $\epsilon^{-1}\, \text{inj}_p\geq s_p$ and $s_q\geq \epsilon\,  \text{inj}_p$ for all $q\in \cB_{\epsilon\,  \text{inj}_p}(p)$.\label{e:case2}
\item There exists a $q\in \cB_{\epsilon\,  \text{inj}_p}(p)$ such that $\epsilon\, \text{inj}_p\geq s_q$.\label{e:case3}
\end{enumerate}

In all three cases we will use that it follows from the curvature estimate for stable minimal surfaces \cite{S2}, \cite{CM2}, that there exists some constant $C$ so that 
\begin{align}
\sup_{\cB_{s_q/2}(q)}|A|^2\leq C\,s_q^{-2}\, .
\end{align}

We will first see that \eqr{e:case2} cannot happen when $\text{inj}_p$ is sufficiently small (i.e., $r_k(p)$ sufficiently large).  To see this observe first that the two assumptions together implies that there exists some constant $C>1$ so that
\begin{align}  \label{e:aprioriineq}
C^{-1}\,  \left(1+\frac{\epsilon}{4}\right)^{-2}\, \text{inj}_p^{-2}\leq \sup_{\cB_{(1+\epsilon/4)\text{inj}_p}(p)}|A|^2\leq \sup_{\cB_{(1+\epsilon/2)\text{inj}_p}(p)}|A|^2\leq C\,\epsilon^{-2} \,\text{inj}_p^{-2}\, .
\end{align}
Here the upper bound follows from the curvature estimate for stable minimal surfaces and the lower bound follows from the definition of $s_p$.  We will now rescale distances in $E_k$ by  $\text{inj}_p^{-1}$ and use that the minimal surface $\Sigma$ lies in the complement of the (open) horosphere $r_k> r_k(p)$, yet intersect it at $p$.  The rescaling of the horosphere can be made as close to a plane in flat $\RR^3$ as we want provided that $\text{inj}_p$ is sufficiently small.  Finally, the rescaling of the minimal surface has a priori curvature estimates by \eqr{e:aprioriineq} and by the same set of inequalities a point in the set where we have a priori curvature estimates has curvatures bounded away from zero.  This however easily contradict that $\Sigma$ is minimal.  Note that the above argument gives a contradiction for all $\epsilon> 0$ fixed provided that $\text{inj}_p$ is sufficiently small depending on the value of $\epsilon$.

This leaves us with possibilities \eqr{e:case1} and \eqr{e:case3}.  If either of these cases happen for some $\epsilon>0$ sufficiently small, then we will show that the minimal surface would be entirely contained in the cusp, thus giving a contradiction.  

Below $D_k$ will be a fundamental domain for the flat torus $\Gamma_k$ which we may, in the usual way, take to be a parallelogram in the plane with one side parallel to the $x$-axis.

In case \eqr{e:case1} we have that the intrinsic ball $\cB_{s_p/2}(p)$ is locally nearly parallel with the horosphere $r_k=r_k(p)$, lies in the complement of the (open) horoball $r_k> r_k(p)$, yet intersect it at $p$.  Moreover, the radius $s_p$ is large compared to $\text{inj}_p^{-1}$ or rather large compared with the diameter of the horosphere.  
Starting at $p$ follow the surface along a curve on the surface that project to a line that is parallel to one of the two sides of the parallelogram in the fundamental domain.  As the corresponding geodesic in the flat torus come back to itself the curve on the surface $\Sigma$ must either come back above, below, or to itself.   It cannot come back above itself as the maximal value for $r_k$ is achieved at $p$; if it came back below itself, then in the other direction along the closed geodesic in the flat torus the curve on the surface would come back above itself which is also a contradiction.   This implies that any curve that project to a curve that is parallel to one of the two sides of the parallelogram in the plane come back to itself on the minimal surface and thus the entire minimal surface is a torus that is entirely contained in a cusp of the $3$-manifold violating the maximum principle.

In case \eqr{e:case3} we argue similarly, but use Theorem \ref{t:mainindex} to show that near $p$ on scale which is large compared with $\text{inj}_p$ (or equivalently large compared with the diameter of the horosphere through $p$) the minimal surface looks (modulo the deck group of the end) like a scaled catenoid that is almost parallel to the horosphere\footnote{In fact, if one use that the minimal surface lies in the complement of the (open) horoball and intersect it at $p$, then one can get away without using Theorem \ref{t:mainindex}.  This is because under this strong assumption one relatively easily get a weak version of Theorem \ref{t:mainindex} that sufficies.}.  Similarly to 
\eqr{e:case1} above once we have this then it follows that the surface is entirely contained in the cusp which is the desired contradiction.  To make this precise we need to show that modulo the deck group of the end $E_k$ the surface in fact look like a scaled catenoid on a large scaled compared with $\text{inj}_p$ and near $p$.   Theorem \ref{t:mainindex} gives that it looks like a scaled catenoid on a scale that is some fixed fraction of $\text{inj}_p$ to show that it is in fact large compare with $\text{inj}_p$ we use that the sheets away from the neck must be very close together.   We can now use that the minimal surface have a priori curvature bounds away from the neck by stability together with that the extension of the sheets lies in the complement of the horoball to get that modulo the deck group of the end it really does look like a scaled catenoid near $p$ on a scale that is large compared with $\text{inj}_p$.  After that we argue similarly to case \eqr{e:case1} to show that as we follow curves that project to lines parallel to one of the two sides in the parallelogram the curve must come back to itself as also in this case we would get a contradiction if it back back below or above itself.  (Here we avoid curves that go into the small
unstable area where curves does not go straight across).  As in the first case this then implies that the minimal surface is entirely contained in a cusp which violates the maximum principle and is the desired contradiction.
\end{proof}


\section{Branched Surfaces}

Given a complete finite volume hyperbolic 3-manifold $N$, let $N'_{[\epsilon,\infty)}$ denote the union of the $\epsilon$-thick part of $N$ together with the Margulis tubes.   In the last section we showed that exists a uniform constant $\rho$ such that any index-1 surface lies in within $N'_{[\rho,\infty)}$.  In what follows, we denote $N'_{[\rho,\infty)}$ by $N'$, so if $N$ is closed, then $N=N'$.  

By Manning \cite{Ma} if $N$ is hyperbolic there exists an algorithm to construct a totally geodesic triangulation and hence compute a lower bound for the injectivity radius.   In general by Breslin \cite{Br} there exists a triangulation $\Delta$ whose tetrahedra have area and angles uniformly bounded below, where the uniformity constant is a function of the injectivity radius of $N$.


\begin{theorem}  \label{eta branched}Let $N$ be a complete hyperbolic 3-manifold with finite volume and $0<\eta<1$.  Then there exists finitely many effectively constructible generic $(1-\eta)$-negatively curved branched surfaces $B_1, \cdots, B_n$ such that any embedded closed index-$\le 1$ surface $S $ is carried by one of these branched surfaces.  \end{theorem}

The proof follows in two steps.
\vskip 8pt
\noindent\emph{Step 1.}  Let $E>0 $ and $\eta>0$.  There are finitely many constructible $-1+\eta$ negatively curved branched surfaces that carry all minimal surfaces in $N'$ with $|A|^2<E$.  

\vskip8pt

\noindent\emph{Proof.}    This routinely follows from standard arguments after recognizing that the bounds on $|A|^2$ and injectivity radius (since $N'$ is compact) imply that any minimal surface $S\subset N'$ is a union of discs of uniform size and geometry.  Further, two such discs that are sufficiently close can be glued together to create a new smooth disc or branched disc (depending on context) such that any disc carried by resulting (branched) surface is nearly minimal and hence is $-1+\eta$ negatively curved, since the ambient manifold N' is hyperbolic.  Here are more details.

Given a lower bound $\delta_0>0$ on injectivity radius, $E>0$, and $\epsilon>0$ there exists computable $r(\delta_0, E)$ and a finite set $\mD=\{D_1, \cdots, D_k\}$ of embedded discs in $N$ of radius $2r$ respectively about points $\{y_1, \cdots, y_k\}$ of mean curvature $<\epsilon/10^{10}$ such that if $x\in  S\subset N'$ is a minimal surface with $|A|^2<E$, then the metric $2r$-ball $B_{2r}(x)$ about $x$ in $S$ is embedded and the subdisc $B_r(x)$ is $\epsilon$-graphical over some $D_i$ where $d(x, y_i)<\epsilon/1000$, i.e. $B_r(x)=\{\exp_z(u(z)\,\nn(z))\, |\, z\in D_i\}$ where $0< u< \epsilon$.   Here we can take $r(\delta_0,E)=\min(\delta_0/8, 1, 1/E)$.

The fundamental gluing lemma is as follows.

\begin{lemma} \label{gluing} Let $N$ be a fixed closed $3$-manifold and $E$ a fixed constant.  Suppose that $\Sigma_1$, $\Sigma_2$ are two embedded minimal surfaces with $|A|^2\leq E$ and that $\Sigma_2$ can be written as a normal graph over $\Sigma_1$; that is, $\Sigma_2=\{\exp_x(u(x)\,\nn(x))\, |\, x\in \Sigma_1\}$ where $0< u< \epsilon$.  Suppose also that the geodesic curvature of $k_g$ of $\partial \Sigma_1$ is bounded by $E$ and $\gamma_1\subset \partial \Sigma_1$ is a component of $\partial\Sigma_1$ parametrized by arclength with the following chord arc bound:

If $x$, $y\in \gamma_1$, then 
\begin{align}
\frac{d_{\gamma_1}(x,y)}{d_N(x,y)}\leq E\, .
\end{align}

Given $\delta>0$ there exists an $\epsilon>0$ depending only on $E$ and $\delta$ so that $\Sigma_1$ and $\Sigma_2$ can be glued together along $\gamma_1$ and so that the glued together surface $\Sigma$ agrees with $\Sigma_1$ and $\Sigma_2$ outside the tubular neighborhood $T_{\delta}(\gamma_1)$.   Moreover, 
\begin{enumerate}
\item The mean curvature $H_{\Sigma}$ of $\Sigma$ satisfies $|H_{\Sigma}|\leq \delta$.\label{e:eqr1}
\item $|A_{\Sigma}|^2\leq E+\delta^2$.\label{e:eqr2}
\end{enumerate}\end{lemma}

\begin{proof}  Let $\zeta$ be a function on $N$ with $1\geq \zeta\geq 0$ so that $\zeta$ is identically one on $\gamma_1$ and vanishes outside $T_{\delta} (\gamma_1)$.  Consider now the surface $\Sigma_+$ given by
\begin{align}
\Sigma_+&=\left\{\exp_x\left(\left(1-\zeta (x)\right)\, u(x)\,\nn(x)\right)\, |\, x\in \Sigma_1\right\}\,  .
\end{align}
It is clear that $\Sigma_+$ agrees with $\Sigma_2$ outside $T_{\delta} (\gamma_1)$ and that $\Sigma_1$ and $\Sigma_+$ intersect along the curve $\gamma_1$. 

We need to see that \eqr{e:eqr1} and \eqr{e:eqr2} holds for $\Sigma_+$.   When $N$ is Euclidean space equation \eqr{e:eqr1} follows from that the normal $\nn_{\Sigma_+}$ of $\Sigma_+$ and its derivative are equal to that of $\nn_{\Sigma_2}$ up to a term that is a bounded constant times $\max u \leq \epsilon$.  The general manifold case follows similarly as well as \eqr{e:eqr2}.\end{proof}

In a similar manner we have the following  

\vskip 10pt

\noindent\textbf{Addendum.}  Let $E>0 $ and $\eta>0$.  There are finitely many constructible $-1+\eta$ negatively curved branched surfaces that carry all closed embedded surfaces in $N'$ with $|A|^2<E$ and mean curvature $\le 1/10^{100}$.  Furthermore, there exists computable $r_0>0$ and $\epsilon_0>0$ so that if $S\subset N'$ is a closed embedded surface with $|A|^2<E$ that contains discs $D_1$ and $D_2$ with $\diam(D_i)<r_0$ and $D_1$ being $\epsilon_0$-graphical over $D_2$, then $S$ is carried by one of the branched surfaces $B$ such that $D_1$ and $D_2$ are identified with the same smooth embedded disc $D\subset B$.  \qed

\vskip 10pt

\noindent\emph{Step 2.}  General Case.

\vskip 8pt

\noindent\emph{Proof.}  If $S$ is index-$0$, then by Schoen's curvature estimae \cite{S2}, \cite{CM2},  there exists a computable uniform upper bound $E_0$ on $|A|^2$ and hence the result follows by Step 1.  By the previous section\footnote{Although the argument in the previous section were by compactness it is well-known how to make it effective; see, for instance, \cite{CM7}.}  given $\epsilon >0$ and $r>0$ there exists a computable $E_1$ such that if $S$ is index-$1$ with $\sup |A|^2>E_1$, then there exists an annulus $An\subset S$, a surface $S_1$  and discs $D_1$, $D_2\subset S_1$ such that $S\setminus \inte(An)=S_1\setminus (\inte(D_1)\cup\inte(D_2))$ and $S_1$ has mean curvature $<1/10^{100}$ and $|A|^2<E_1+1$.  Finally, $\diam (D_i)<r_0$ and $D_1$ is $\epsilon_0$-graphical over $D_2$ where $r_0$ and $\epsilon_0$ are as in the Addendum.

The desired branched surfaces are constructed as follows.  Let $\mB=\{B_1, \cdots, B_n\}$ be the branched surfaces arising from the Addendum.  For each sector $F$ of $B_i$, construct a branched surface $B_i^F$, by picking an embedded disc $D_F\subset F$ about a point $f\in \inte(F)$, splitting $B_i$ along $D_F$ to create a new branched surface $\hat B_i^F$ with a small $D^2\times I$ closed complementary region $J$ and inserting a small annulus within $J$.  See Figure 1 at the end of the pdf version.  Note that Lemma \ref{gluing} implied to $D_f$ and a graphical copy implies that $\hat B_i^F$ can be constructed to have mean curvature and $|A|^2$ bounded above by that of $B_i^F$ up to any predetermined $\delta>0$.  Similarly, we can pass from $ \hat B_i^F$ to $B_i^F$ by gluing a rescaled catenoid into $J$ so that the resulting branched surface is $(1-\eta)$-negatively curved.  The previous paragraph implies that any index-1 surface is carried by one of these branched surfaces.\qed

\section{Effective polynomial growth}

\begin{definition}  Let $S$ be a Riemann surface and $f:\BR\to \BR$.  We say that $S$ has growth-$f$ if for each $x\in S$ and $r>0, \area(N_S(x,r))\le f(r)$.  If $f$ is a polynomial, then we say that $S$ has \emph{polynomial growth}.\end{definition}   

In 1975 Joe Plante  proved (Theorem 6.3 \cite {Pl}) that any leaf of a transversely oriented codimension-1 foliation without holonomy has polynomial growth.  In this section we apply Plante's argument to show that there is a single polynomial that works for all measured laminations carried by a fixed branched surface.  More precisely we have

\begin{theorem}   Let 	$B$ be a branched surface embedded in the Riemannian 3-manifold $M$.  There exists an effectively constructible polynomial $p(B)$, such that if $S$ is a leaf of a measured lamination carried by $B$, then the growth of $S$ is bounded by $p$.\end{theorem}

\begin{proof}  Let $N(B)$ be a fibered neighborhood of $B$.  Since the growth rate of a surface is non decreasing upon passing to covering spaces and $B$ has an effectively constructible transversely orientable 2-fold cover by Proposition \ref{2-fold}, it suffices to consider the case that $B$ is transversely orientable and hence the vertical  fibering $\mV$ of $N(B)$ is orientable.  Assume each element of $\mV$ has length $\le 1$.  It suffices to show the following three steps.  \vskip8pt

\noindent\emph{Step 1}.  There exists an effectively constructible polynomial $h$ such that the number of distinct non torsion elements in $H_1(N(B))$ represented by closed curves of length at most $n$ is bounded by $h(n)$.\vskip 8pt

\noindent\emph{Proof of Step 1.} Let $\{S_1, \cdots, S_b\}$ be a set of properly embedded oriented surfaces in $N(B)$ representing a basis for  $H_2(N(B),\partial N(B))$.  By Poincare - Lefshetz duality non torsion elements of $H_1(N(B))$ are determined by their algebraic intersection number with these surfaces.  Let $c>0$ be such that if $\alpha\subset N(B)$ is a closed essential oriented curve, then for each $i$, $\length(\alpha)\ge c |<\alpha, S_i>|$.  Then  $h(n)=(2n/c)^b$ is a bounding function.\qed\vskip8pt

\noindent\emph{Step 2.}  Let $s_1, \cdots, s_k$ be the sectors of $B$ and $I_1, \cdots, I_k \in \mV$ be $I$-fibers respectively passing through these  sectors.   To prove the theorem it suffices to show that for each $ i$ there exists a polynomial $h_i$ such that if $\mu$ is a measured lamination carried by $B$, $L$ is a leaf of $\mu$ and $x\in L$, then $|N_L(x,r))\cap I_i|\le h_i(r)$.\vskip8pt

\noindent\emph{Proof of Step 2.}  This follows from the fact that there are a finite number of sectors of uniformly bounded area and diameter.  \qed\vskip 8pt

\noindent\emph{Step 3.}  Let $L$ be a leaf of a measured lamination $\mu$ carried by $N(B)$ and let $x\in L$.  If $N_L(x,r)\cap I_i=\{x_0, x_1\cdots, x_m\}$, then there exist closed curves  $\gamma_1, \cdots, \gamma_m$ in $N(B)$ such that for all $j$, $\length(\gamma_j)\le 2r+1$ and the $\gamma_j$'s represent distinct elements of $H_1(N(B))$.\vskip 8pt

\noindent\emph{Proof of Step 3.}  Since $B$ is transversely orientable,  $\mu$ determines an element $[\mu]\in H^1(N(B), \BR)$  by \emph{integrating}, with sign, closed curves over $\mu$.   For $j\in \{0,1,  \cdots, m\} $ let $\gamma_j$ be an oriented curve of length $\le 2r+1$ obtained by concatenating $a_j, b_j, c_j$, where $a_j$ is a path in $L$ of $\length \le r$ from $x$ to $x_0$; $b_j$ is a path in $\mI_i $ of length $\le 1$ from $x_0$ to $x_i$ and $c_i$ is a path of length $\le r$ in $L$ from $x_j$ to $x$.  Since the values $[\mu](\gamma_i)$ are distinct the result follows.  \end{proof}



\section{Horizontally large branched surfaces}

\begin{definition}  A branched surface $B$ is \emph{horizontally large} if no component of $\partial_hN(B)$ is a disc or an annulus. \end{definition} 


\begin{proposition}  \label{r splitting} Let $B$ be an $\eta$-negatively curved branched surface in the Riemannian 3-manifold M that fully carries a surface and let $r>0$.  Then there exists effectively constructible branched surfaces $B_1, \cdots, B_n$ obtained by regularly splitting $B$ such that every surface carried by $B$ is carried by some $B_i$ and each $B_i$ fully carries a surface.  Furthermore, if $E$ is  a subbranched surface of a regular splitting of some $B_i$,  then for each component $H$ of $\partial_h N(E)$ there exists $x\in H$ such that $N_H(x,r)\subset \inte(H)$.\end{proposition}

\begin{corollary}  \label{no annuli} If $ B$ is an $\eta$-negatively curved branched surface in the 3-manifold $M$ that fully carries a surface, then there exists effectively constructible horizontally large branched surfaces $B_1, \cdots, B_n$ obtained by regularly splitting $B$ such that every surface carried by $B$ is carried by some $B_i$ and each $B_j$ fully carries a surface.  Furthermore, any subbranched surface $B'$ of any branched surface obtained by regularly splitting a $B_i$ either is horizontally large or does not fully carry a surface. \end{corollary}

\noindent\emph{Proof of Corollary.}  Applying Theorem \ref{plante} to the $\eta$-negatively curved branched surface $B$ we conclude that there exists an effectively computable  $s_0>0$ such that any surface $S$ carried by $B$ has injectivity radius at most $s_0$ at all points of $S$.  This follows by comparing the exponential growth rate of discs in negatively curved surfaces and the uniform polynomial growth rate of $S$.  On the other hand any surface carried by $B$ has a uniform lower bound of injectivity radius.  Thus, there exist  uniform $s_1, s_2$, depending only on $s_0$, $\eta$ and $B$ such that if $x\in S$, then $N_S(x, s_1)$ contains a closed geodesic and $N_S(x,s_2)$ contains a non empty $\pi_1$-injective subsurface $T$ with geodesic boundary.  In particular $T$ is not an annulus or disc.  

Now apply the Proposition with $r=s_2$ to obtain the branched surfaces $B_1, \cdots, B_n$.     Thus if $E$ is a subbranched surface of some branched surface $E'$ obtained by regularly splitting $B_i$ and $H$ is a component of $\partial_h N(E)$, then for some $x\in H$, $N_H(x, s_2)\subset \inte(H)$.  If $E$ fully carries a surface $S$, then after possibly doubling $S$ we can assume that $H\subset S$ and hence by the first paragraph  $H$ is neither a disc nor an annulus.  \qed

\vskip 10 pt

\noindent\emph{Proof of Proposition.}  Fix $r>0$. It is routine to regularly split  $B$ into $B_1(r), \cdots, B_{n_r}(r)$ such that every surface carried by $B$ is carried by some $B_i(r)$, each $B_j(r)$ fully carries a surface and the following additional property holds. For each $i$, under the natural inclusion of $\partial_h N(B)$   into $\partial_h N(B_i(r))$,  the metric $2r$-ball about $\partial_h N(B)$ in $\partial_h N(B_i(r))$ is contained in the interior of $\partial_h N(B_i(r))$.  Note that any regular splitting of $B_i(r)$ continues to retain this property.

Let $E$ be a subranched surface of a regular splitting $B'$ of some $B_i(r)$.  We show that each component $H$ of $\partial_h N(E)$ contains a point $x$ such that $N_H(x,r)\subset \inte(H)$.  If it fails for $H$, then let $F$ be a component of $\partial_h N(B')$ that corresponds to a point $y\in H$.  We obtain a contradiction by showing that $F$ does not arise from a component of  $\partial_h N(B)$ contradicting the fact that $B'$ is the result of a regular splitting of $B$.  View $N(E)\subset N(B')\subset N(B)$ so that the various $I$-fibers of one lie in the interior of the next.    We will assume that $ B$ is transversely orientable for the proof in the general case is similar.    It suffices to show that there exists an embedding $f:H\times I\to N(B)$ such that $f|H\times 0$ is a homeomorphism onto $H$, $f|\partial H\times I$ is a homeomorphism onto some components of $\partial_v N(E)$, $f|F\times 1$ is a homeomorphism onto a component $H_1$ of $\partial_h N(E)$ and each $f(x\times I)$ embeds into an $I$-fiber of $N(B)$.   Therefore $F\subset f(H\times I)$ and hence does not correspond to any component of $\partial_h N(B)$ under a regular splitting.  Let $I_x\subset N(B)$ denote the $I$-fiber that contains $x$.  It suffices to show that if $x\in H$, then there exists a subinterval $J_x\subset I_x$ such that $x\in J_x$ and $J_x\cap N(E)=\partial J_x$.  

 Fix $x\in \inte(H)$.  We show that   $J_x$ exists.   Give $\partial_h N(E)$ and $\partial_h N(B')$ the Riemannian metrics induced by their projections into $B$ by contracting each fiber of $N(B)$ to a point.   By construction of $B'$, if $z\in \partial_h N(B')$ and $d_{\partial_h N(B')}(z, \partial_v N(B' )\le 2r$, then either $z\in \partial_v N(B')$ or there exists a subinterval $K_z\subset I_z$ such that $K_z\cap N(B')=\partial K_z$.  
 For $x\in \inte(H)$ let $L_x$ denote the maximal subinterval of $I_x$ such that $L_x\cap N(E)\subset \partial L_x$ and $\partial L_x\subset N(E)\cup \partial_h N(B)$.  Since for $x\in \inte(H)$ and $t\in L_x\cap \partial_h N(B')$, $d_{\partial_h N(B')}(t, \partial_v N(B')\le d_{\partial_h N(E)}(x,\partial_v N(E))$ it follows that $J_x$ exists and equals $L_x$.  Indeed, if $\alpha\subset H$ is a path of length $\le 2r$ from $x$ to $\partial H$, then this assertion evidently holds for points of $\alpha$ near $\partial H$ and then readily extends to all of $\alpha$.
\qed\vskip 10pt

\section{proof of Theorem \ref{main}}


\begin{definition}  A 1-sided Heegaard surface $S\subset M$ is a closed 1-sided surface whose complement is a handlebody.   $S$ gives rise to a Heegaard surface $T$ by passing to a 2-fold cover and attaching a neck between locally parallel sheets.  We call $S$ and $T$ \emph{associated} surfaces.\end{definition}

We record the following well known result.  Recall that a surface $S$ is \emph{compressible} if there exists an \emph{embedded} disc $D$ such that $D\cap S=\partial D$ and $\partial D$ is an essential curve in $S$. Also recall Remark \ref{compressing}.

\begin{lemma}  \label{1-sided}  If $S$ is a compressible 1-sided Heegaard surface, then the associated Heegaard surface is weakly reducible.  If $\chi(S)<0$ and there exists an immersed compressing disc $D$ for $S$ that restricts to an essential compressing disc for $\partial N(S)$ such that $\partial D$ has a single point of self intersection, then the associated Heegaard splitting $T$ is weakly reducible.\qed\end{lemma}

\begin{theorem} \label{technical}  Let $f:B\to M$ be an embedding of an $\eta$-negatively curved branched surface into the 3-manifold $M$.  Either $B$ carries an incompressible surface or there exists a computable $C(f,\eta)\in \BN$ such that if $S$ is a Heegaard surface of genus $\ge C$ and either $S$ or an associated  1-sided surface  is carried by $B$  then $S$ is weakly reducible. \end{theorem}



\noindent\emph{Proof of Theorem \ref{main}.} By Theorem \ref{negative branched} for any $-1<\eta<0$ we can effectively construct finitely many $\eta$-negatively curved branched surfaces,  such that all index-$\le 1$ minimal surfaces are carried by these branched surfaces. Fix $\eta=-1/2$.  As announced in \cite{PR} any strongly irreducible Heegaard surface is isotopic to one that is either index-$\le 1$ or has an associated 1-sided surface of index-0;  thus it or an associate is carried by one of these branched surfaces.  Now apply Theorem \ref{technical} to each of these branched surfaces.  If $M$ is non Haken, then the first conclusion never arises. Let $G(M)$ be the maximum of $C(f,\eta)$ over all these branched surfaces.\qed\vskip10pt

\noindent\emph{Proof of Theorem \ref{technical}.}  By Haken \cite{Ha} (stated in the context of normal surfaces) there exists a finite set  $F_1, \cdots, F_k$ of \emph{fundamental surfaces} carried by $B$ such that any surface $S$ carried by $B$ is of the form $n_1F_1+\cdots +n_kF_k$ where $n_i\in \BZ_{\ge 0}$.  (Here addition is by cut and paste in the manner compatible with being carried by $B$.)    Thus the subbranched surfaces that fully carry a surface are finite and enumerable.  By \cite{JO} one can algorithmically decide if a branched surface is essential.    Our argument will require us to consider finitely many subbranched surfaces of branched surfaces obtained by splitting $B$.  Since we can apply \cite{JO} at each stage, we will always assume that all branched surfaces under consideration are not essential.   By Corollary \ref{no annuli}  it suffices to consider the case that $B$ is horizontally large and fully carries a surface and that any subbranched surface of a regular splitting of $B$ that fully carries a surface is also horizontally large.    Being $ \eta$-negatively curved, $\chi(F_i)<0$ for all $i$.  Since $\chi(S)=\sum_{i=1}^k n_i\chi(F_i)$  it suffices to show that  there exists a computable number $K^*$ such that if $S=n_1F_1+\cdots +n_kF_k, S$ is a Heegaard surface and for some $i, n_i\ge K^*$, then $S$ is weakly reducible.

Let $S$ be any surface of the form $n_1F_1+\cdots  + n_kF_k$.  After reordering and deleting the terms with $n_i=0$ we can assume that $1\le n_1\le n_2\le\cdots\le n_m$ where $m\le k$.  By the first paragraph we can assume that the subbranched surface $B_S$ that fully carries $S$ is horizontally large.   Fix $K\ge 2$.  (Its actual value will be computed later.)  Define $K_i=K^{10^i}$.  Let $p$ be the largest value such that $n_p<K_p$ or take $p=0$ if no such value exists.  Define  $S_1=n_1F_1+\cdots + n_pF_p$ and $S_2=n_{p+1}F_{p+1}+\cdots +n_m F_m$.  We shall see in Lemma \ref{p>1} that if $S$ is a strongly irreducible Heegaard surface,  then $n_1=1$ and hence $p\ge 1$.   Ultimately we will take $K^*=K_k$, thus our surfaces  of interest will satisfy $S_1,S_2\neq\emptyset$.  Let $B_2$ be the subbranched surface carried by $S_2$.  Note that there are only finitely many possibilities for $B_S$ and $B_2$.

\begin{definition} \label{germ} Let $A=\alpha\times I$ be a germ of an immersed smooth annulus in $M$ 
with $\alpha\times I\cap B_2=\alpha\times 0$ and $\alpha\times (0,1]$ is embedded and transverse to 
$B_S$.   Denote $\alpha\times 0$ by $\alpha$. Let $s_1, s_2, \cdots, s_p$ denote the germs of arcs of 
$B_S\cap A$ with one endpoint on $\partial A$.  (There may be no such germs.)  These germs 
correspond to cusps $c_1, \cdots, c_p$ at $\alpha \times 0$.  Fix an orientation on $\alpha$.     If all the 
germs attach in the same manner with respect to the orientation (i.e. the cusps point in the same 
direction), then we say that $B_S$ attaches to $B_2$ \emph{consistently} along $\alpha$  on the $A$-
side otherwise $B_S$ attaches \emph{inconsistently}.  Since consistent attachment is independent of the 
choice of germ, it makes sense to talk about consistent/inconsistent attachment once the side of $\alpha$ 
is chosen. \end{definition}

The Claim  below roughly says that after replacing $(B_S, B_2)$ by another pair obtained by splitting $B$ and passing to certain subbranched surfaces the resulting pair and its associated data satisfy many nice properties.   Furthermore, there are only finitely (pairs, data) that need to be considered.

\vskip 10pt


\noindent\emph{Claim.}  \label{clean up} There exists an effectively enumerable finite set of pairs of horizontally large branched surfaces $(E_1, L_1), \cdots, (E_u, L_u)$ with each $L_i$ a subbranched surface of $E_i$  such that given a surface $S$ carried by $B$ with $S=S_1+S_2$ as above  and $S_1, S_2\neq \emptyset$, there exists a $j$ such that \vskip 8pt

1) $E_j$ fully carries $S$ and $L_j$ fully carries $S_2$.\smallskip

2) Neither $E_j$  nor $L_j$ have discs of contact.\smallskip

3) For every sector $\sigma$ of $L_j$, $w_\sigma(S_2)\ge K_{p+1}$, where $w_\sigma(S_2)$ is the weight of $S_2$ in $\sigma$.  \smallskip

4) For every sector $\tau$ of $E_j\setminus L_j, w_\tau(S_1)\le C K_p$ for some $C$ that depends only on $B$.\smallskip

5) If $L_j$  is compressible, there corresponds a compressing disc $D_j$  for $L_j$  transverse to $E_j$ and   $P'_j \subset \partial_h N(L_j)$ a compact essential  pair of pants disjoint from $\partial_v N(L_j)$  such that $\partial D_j \subset \partial P'_j$ and for each component $\beta'$ of $\partial P'_j$,\   $E_j$ attaches to $L_j$ consistently along  the $D_j$-side of $\pi(\beta')$. \smallskip 

6)  If $S$ is a strongly irreducible Heegaard surface with complementary handlebodies $H_1, H_2$, then there exist compressing discs $A^j_1, A^j_2$ for $E_j$ independent of $S$, such that $A^j_1\subset H_1, A^j_2\subset H_2$ with $\partial A^j_1\cup \partial A^j_2 \subset \inte(\sigma)$ where $\sigma$ is a sector of $E_j\setminus L_j$.  \smallskip

\begin{remark}  We do not need an analogue of 6) when $S$ is a 1-sided Heegaard surface.\end{remark}

\noindent\emph{Proof of Claim.}  Define $B_S$ and $B_2$ as above.  Since $B_2$ fully carries $S_2$ condition 3) holds for $B_2$.  By the linear isoperimetric inequality for $\eta$-negatively curved surfaces, $B_2$ has only finitely many discs of contact and their areas are uniformly bounded, thus after a uniformly bounded amount of splitting we obtain a branched surface $B^1_2$ with out such discs.  Observe that  condition 3) continues to hold.  As there are only finitely many possibilities for $B_2$, there are only finitely many $B^1_2$'s that can arise.  Our $L_j$'s will be among these branched surfaces.

A preliminary uniformly bounded amount of $S$-splitting transforms $B_S$ to a branched surface $B^1_S$ containing $B^1_2$ as a subbranched surface.   Thus while $B_S$ depends on $S$, the number of possible resulting $B^1_S$'s as $S$ varies is uniformly bounded above independent of $K$.  Since a fundamental surface goes over a sector of $B$ a uniformly bounded number $C_0$ of times, it follows that for each sector $\tau$ of $B_S\setminus B_2$, $w_\tau(S)< k C_0 K_p$.  Since $B^1_S$ is obtained by splitting, a similar fact holds for it. Since each of $B^1_S$ and $B^1_2$ carry surfaces, each is horizontally large.    Let $B^2_S$ be obtained from $B^1_S$ by splitting along all its discs of contact.  Note that $B^1_2$ continues to be a subbranched surface of $B^2_S$ and the number of possibilities for $B^2_S$ is uniformly bounded.

If $B^1_2$ is not essential, then either $B^1_2$ compresses or $B^1_2$ has a monogon.  Since $B^1_2$ is horizontally large the latter implies the former.    Fix a compressing disc $D$ for $B^1_2$.  By \cite{JO} compressing discs are algorithmically findable.   Since $B^1_2$ is horizontally large there exists an immersed pants $P\subset B^1_2$ that lifts to an embedded essential pants $P'\subset\partial_h N(B^1_2)$ with $\partial D$ a component of $\partial P'$. 

If the germs of arcs of $D\cap B^2_S$ near $\partial D$ are not consistent, then we can find adjacent germs whose cusps point towards each other.  After an $S$-splitting of $B^2_S$ supported  near the segment of $\partial D$ connecting these cusps we obtain a branched surface $B^3_S$ with fewer germs.  Thus after a uniform amount of splitting we obtain a branched surface $B^4_S$ with consistent germs.  Again after  further splitting to eliminate contact discs we can assume that $B^4_S$ and $B^1_2$ continue to satisfy conditions 1-5. A similar argument deals with the other components of $\partial P'$.   \vskip 8pt

\begin{lemma}  \label{p>1}  If $S$ is a strongly irreducible Heegaard surface or a 1-sided associate to a strongly irreducible Heegaard surface, then $n_1=1$. \end{lemma}

\begin{proof}  If $S$ is 1-sided and $n_1>1$, then a compressing disc for $ \partial_h N(B^4_S)$ gives rise to a compressing disc for $S$ and hence the associated Heegaard surface $T$ to $S$ is weakly reducible.  Now assume that $S$ is a Heegaard surface.  Since $B^4_S$ is horizontally large, $B^4_S$ essentially compresses to both the $H_1$ and $H_2$ sides respectively along discs $D_1$ and $D_2$.  Since $B^4_S$ has no discs of contact, these $D_i$'s correspond to essential compressions of $S$.  If $n_1\ge 2$, then the weight of $S$ on each sector of $B^4_2$ is at least two which implies that $\partial D_1\cap \partial D_2=\emptyset$ and hence $S$ is weakly reducible. \end{proof}


Now suppose that $S$ is a strongly irreducible Heegaard surface.  As $B^4_S$ is horizontally large without discs of contact there exist essential compressing discs $A_1,A_2$ for $B^4_S$ lying respectively in $H_1, H_2$.     We can always choose a fixed $ A_1$ independent of $S$ and then choose $A_2$ to be one of a set of at most $|C(N(B^4_s))|-1$ elements.  Since $S$ is strongly irreducible,  $A_1\cap A_2\neq\emptyset$ and hence both intersect the same sector $\sigma' $ with $w_{\sigma'}(S)=1$.  It follows that after a uniform amount of splitting of $B^4_S$ we obtain $B^5_S$ with $\partial A_1\cup \partial A_2\subset\inte(\sigma) $ and $w_\sigma(S)=1$.  Here $\sigma$ is the sector descended from $\sigma'$.  After further splitting we obtain $B^6_S$ satisfying conditions 1-6).  Note that if $B^6_S$ carries another strongly irreducible surface $S'$, then as in the proof of Lemma \ref{p>1}, $w_\sigma(S')=1$, hence $A_1, A_2$ suffice for $S'$.  Thus a given $B^4_S$ produces finitely many $B^6_S$'s satisfying the conclusion of 6).  To clarify, the branched surface $B^4$ gives rise to finitely many pairs $A_1, A_2$ and each pair gives rise to a $B^6$, thus the $A^j_1, A^j_2$ in 6) are independent of $S$.  This completes the proof of the claim.\qed\vskip10pt

\begin{remark}  The sector $\sigma$ is the analogue in our setting of the almost normal disc that appears in Li's work. \end{remark}

Again suppose that $S =S_1+S_2$ either is a strongly irreducible Heegaard surface or is 1-sided with a strongly irreducible associate.  Assume that $S$ corresponds  to the branched surfaces $E_j, L_j$, the disc $D_j$ and  the pants $P'_j$ and the whole package satisfies the conclusions 1)-5) of the Claim.  If $S$ is 2-sided assume in addition that  6) also holds.   In what follows we will assume that the projection $\pi|P'_j$ is an embedding, i.e. the image $P_j\subset $ is an embedded surface.  The proof in the general case is very similar by noting that in the argument that follows all the \emph{action} will lie to the $D_j$-side of $P_j$.   To simplify and abuse the notation we  now denote $E_j, L_j, P_j, D_j, A^j_1, A^j_2$ respectively by $B_S, B_2, P, D, A_1, A_2$.

\begin{lemma} \label{cusps} If $S$ either is strongly irreducible or is the 1-sided associate to a strongly irreducible surface, and $D'$ is a compressing disc for $B_2$, then there exist non trivial germs of arcs of $D'\cap B_S$ with endpoints in $\partial D'$.\end{lemma}

\begin{proof}  If not, then $S\cap \piinv(D')$ contains at least $K_{p+1}/2 -1\ge 1$ parallel disjoint circles lying   \emph{interior} to $N(B_S)$ (see Definition \ref{basic}).  The 1-sided case now follows by Lemma \ref{1-sided}.   Since $B_S$ is horizontally large without discs of contact, all these curves are essential in $S$.    By Scharleman's \cite{Sc} nesting property if $S$ is strongly irreducible, then a given  interior curve bounds a disc $D_1$ in one of $H_1$ or $H_2$.  Suppose that $D_1\subset H_1$.  Since $B_S$ is horizontally large and without discs of contact, there exists an essential compressing disc $D_2\subset H_2$  with $\partial D_2\subset \partial_h N(B_S)$ and hence disjoint from $\partial D_1$.   \end{proof}

Denote the components of $\partial P$ by $\partial D$, $\alpha$ and $\beta$.   Orient $\partial D$ in the 
direction that the cusps on the $D$-side of $\partial D$ point.  Orient $\alpha, \beta$ and $P$ so that $
\partial P=\alpha+\beta+\partial D$.      Let $d_0\in \partial D$ (resp.  $a_0\in\alpha$, $b_0\in \beta$) be 
disjoint from $b(B_S)$ and  $I_D$ (resp. $I_\alpha$, $I_\beta$) be the $I$-fiber of $N(B_S)$  through 
$d_0$ (resp. $a_0, b_0$).   
Let $\gamma\subset P$( resp. $\delta\subset P$) be an embedded arc from $d_0$ to $a_0$ (resp. $d_0$ to 
$b_0$).   Denote the points of $S\cap I_{\partial D}$  (resp. $S\cap I_\alpha$, $S\cap I_\beta$) by $d_1, \cdots, 
d_n$ (resp. $a_1, \cdots a_p$; $b_1, \cdots b_q$) where  $d_1$ is the first point on the $D$-side of $I_D$ 
and the others appear in order.  Similarly order $S\cap I_\alpha$ and $S\cap I_\beta$.  There exists $d,a,b\in 
\BZ$ such that the holonomy of $S$ around $\piinv(\partial D)$ (resp. $\piinv( \alpha)$, $\piinv(\beta)$) is given 
by $d_i$ goes to $d_{i+d}$ (resp. $a_i$ goes to $a_{i+a}$, $b_i$ goes to $b_{i+b})$ provided that it is defined.  
By the choice of orientation on $\partial D$, $d>0$.  Since $a+b+d=0$ we can assume, after possibly switching $\alpha$ with $\beta$, that 
$a<0$.  

If $b=0$, then by Claim 5), there is no branching of $B_S$ on the $D$-side of $ \beta$.  Thus there exists an embedded annulus $\beta\times [0,1]  \subset N(B_S)$ such that $\pi(\beta\times I)=\beta, \beta\times 0\subset \partial_h N(B_S)$ and $S\cap \beta\times I$ is a union  of at least $K_p/2-1$ parallel circles $\beta_0=\beta\times 0, \beta_1, \cdots, \beta_n$ such that, with the exception of  $\beta\times 0$, each is interior to $N(B_S)$.

\begin{lemma} \label{circles essential} No  $\beta_i$ bounds a disc in $S$ and no distinct pair of $\beta_i$'s bound an annulus.\end{lemma}

\begin{proof} If $\beta_i, i>0$ bounds a disc in $S$, then since $S$ has no discs of contact and is horizontally large it follows that $\beta_{i-1}$ does too and hence by induction so does $\beta_0$.  Since $B_S$ has no discs of contact,  this implies that $\beta_0$ bounds a disc in $\partial_h N(B_S)$ and hence in $C(N(B_2))$, contradicting Lemma \ref{cusps}.  A similar type of argument establishes the second claim.\end{proof}

\begin{lemma} \label{holonomy}  There exists  uniform $c_0, c_1\in\BN$ such that if $K>c_1$ and  $c_0 K_p\le i \le K_{p+1}-c_0 K_p$, then the curve $\tau_i\subset S$ is an essential simple closed curve in $S$.  Furthermore, no pair $\tau_i, \tau_{i+1}$ are homotopic in $S$.  Here $\tau_i$ is the curve that starts at $d_i$ and follows the path in $S$ that projects to  $\partial D^{-a} *\gamma*\alpha^d*\bar\gamma$.  \end{lemma}

\begin{remark}  Uniform means independent of $(E_j, L_j, D_j, P_j, A_1, A_2)$.\end{remark}

\begin{proof}  By Claim 4) if the arc $\gamma$ crosses the branch locus $r$ times, then if one starts at the point $d_j$ and follows a path in $S$ that projects to $\gamma$, then one ends at $a_m$ where $|j-m|\le r CK_p$.  Similarly $d< s_0 C  K_p$ (resp. $a< s_1 C K_p$) where $s_0$ (resp. $s_1$) is the number of times $\partial D$ (resp. $\alpha$) crosses the branch locus on the $D$-side. By Claim 3) $\min{|I_D\cap S|, |I_\alpha\cap S|}\ge K_{p+1}$.  Since $K_{p+1}/(a d+r C K_p)\to \infty$ as $K\to \infty$ it follows  that for computable  and sufficiently large  $c_0$ and $K$ the curve $\tau_i$ is well defined.  By construction it is simple.

We first consider the case $b=0$.    Let  $\sigma'\subset P$ be a properly embedded path in $P$ with $\partial \sigma'\subset \beta
$ crossing $\pi(\gamma)$ transversely once and intersecting $b(B'_S)$ a minimal number of times.  If $\sigma_{i+1} \subset S $ is 
such that $\pi(\sigma_{i+1})=\sigma'$ and passes near $d_{i+1}$, then $|\sigma\cap \tau_{i+1}|=1$  and $\sigma\cap \tau_i=
\emptyset$.  Since $\partial \sigma$ lies in  circle components of $\piinv(\beta)\cap S$, it follows using Lemma  \ref{circles essential}  that 
$\tau_{i+1}$ is essential in $S$ and is not homotopic in $S$ to $\tau_{i}$.  Provided that $c_0, c_1$ and $K$ are sufficiently 
large, then for $i$ in the range of the Lemma both $\sigma_{i+1}$ and $\sigma_i$ exist with the desired properties.

We now show that if $b\neq 0$ and $c_0, c_1$ and $K$ are sufficiently large and $i$ is in the desired range, then  $\tau_i $ is non 
separating in $S$ and not homotopic to $\tau_{i+1}$.  If $a=-1$, then $\tau_i$ can be isotoped to intersect $\tau_{i+1}$ in one point 
and hence they are homologically independent. By construction $d\ge 2$ and  $d\neq -a$ since $b\neq 0$.  If $d<-a$ (resp. $d>-a$), then there is 
curve $\kappa_{i+1}\subset S$ starting near $d_{i+1}$ such that $\pi(\kappa_{i+1})$ is homotopic to $\partial D^{-1} *\gamma*
\alpha^{-1}*\bar \gamma*\partial D *\gamma*\alpha*\bar \gamma$ (resp. $\gamma * \alpha* \bar\gamma*\partial D*\gamma *
\alpha^{-1}*\bar \gamma* \partial D^{-1})$ such that $\kappa_{i+1}\cap \tau_i=\emptyset$ and $\kappa_{i+1}$ intersects $\tau_{i
+1}$ once transversely.  \end{proof}




Let $B^a_S$ denote the (possibly nongeneric) branched surface obtained by maximally regularly $S$-splitting $E_j$ without effecting $B_2$, i.e. regularly $S$-split as much as possible away from $B_2$ so that  $b(B^a_S)\subset B_2$ is contained in $B_2$.  Next split $B^a_S$ to remove all discs of contact and call the resulting branched surface $B'_S$.  Note that $B'_S$ carries $B_2$.  Since each non simply connected component of $\mI(S)\cap \piinv(P)$ is untouched in the passage from $B_S$ to $B'_S$ it follows that Lemma \ref{holonomy} continues to hold with $B_S$ replaced by $B'_S$.  We abuse notation by letting $A_1, A_2$ denote the compressing discs for $B'_S$ descending from $B_S$.  

\begin{lemma} \label{no circles} After an  isotopy of $B'_S$ supported away from $B_2$ we can assume that $B'_S\cap \inte(D)$ contains no simple closed curves.\end{lemma}

\begin{proof}   Suppose that $\kappa$ is simple closed curve in $B'_S\cap \inte(D)$, then since $\kappa\cap b(S')=\emptyset$,  $\piinv(\kappa)\cap S$ is a union of parallel circles.   Since $B'_S$ is horizontally large without discs of contact it follows that either each such curve is essential in $S$ or $\kappa$ bounds an embedded disc $E\subset B'_S$ with $E\cap b(B'_S)=\emptyset$.  In the latter case using the irreducibility of $M$, $\kappa$ and perhaps other closed curves can be isotoped off of $D$ without introducing new curves of intersection. Thus we can assume that all the remaining closed curves correspond to families of essential curves in $S$.  
If $S$ is 1-sided and some closed curve of $B'_S\cap \inte(D)$ exists, then by considering an innermost such curve in $D$ we obtain a contradiction to Lemma \ref{1-sided}.

Now assume that $S$ is 2-sided and $\kappa_1,\cdots \kappa_n$ denote the simple closed curves of $B'_S\cap \inte(D)$.  By \cite{Sc}, strong irreducibility and possibly switching $H_1$ with $H_2$ each component of each $\piinv(\kappa_i)$ bounds a disc in $H_1$.  Note that for every  $i$, $|\piinv(\kappa_i)|=1$ else one of these discs is disjoint from $A_2$.  Using these discs in $H_1$ and the irreducibility of $M$ we can assume that each $\kappa_i$ bounds a disc in $D\cap H_1$, in particular is innermost in $D$.  

An outermost disc argument involving $A_2$ enables us to eliminate the $\kappa_i'$s as follows.   By irreducibility of $M$ we can assume that  $A_2\cap D$ consists of arcs.    An arc $\alpha$ of $A_2\cap D$ outermost in $A_2$ bounds a subdisc $F\subset A_2$.  Use $F$ to isotope $B'_S$ to $\hat B_S$ so that if $\hat A_2$ is the resulting disc, then $\hat A_2\cap D=A_2\setminus \alpha$.  Since by Claim \ref{clean up}  $A_2\cap B_2=\emptyset$, $B_2$ does not move in this isotopy.    Observe that either $\hat B_S \cap \inte D$ contains no simple closed curves or each such curve bounds a disc in $H_1$.  There are several cases to check depending on the value of $z=|\alpha\cap (\cup \kappa_i)|$ being $0,1$ or $2$.  Now isotope $\hat B$ as above so that each closed component of $\hat B\cap \inte D$ bounds a disc in $D\cap H_1$.  Thus by induction we either remove all closed curves of intersection or eventually isotope $A_2$ off of $D$  with all closed curves of intersection  bounding essential discs in $H_1$, contradicting strong irreducibility.\end{proof}

Lemma \ref{cusps} now implies that there exists an embedded monogon $m_1\subset D$ with $m_1\cap N(B'_S)=\partial m_1$.  For $c_0 K_p+1\le i \le K_{p+1}-c_0 K_p-1$, let $\tau_i$ be the curve constructed in Lemma \ref{holonomy} and $Z^1_i, Z^2_i$ the vertical annuli in $\mI(S)$ with one boundary component $\tau_i$.   If $S$ is 2-sided assume that $Z^1_i\subset H_1$.  Since $d\ge 2$ each $Z^i_j$ will be embedded unless $a=-1$, in which case $S$ does not separate and hence is 1-sided and after a small isotopy $Z^i_j$ has a single point of self intersection.  In a manner similar to  \cite{Li2} use $m_1$ and a $Z^i_j$  to create an  compressing disc for $S$.   To start with attach a \emph{tail} to $m_1$ to extend it to a disc $m'_1$ with $\partial m'_1\subset S\cup I_D$.    The tail is an embedded $I\times I \subset \piinv(\partial D)$, a union of $\mI(S)$ fibers, with $I\times I\cap m_1=0\times I$ and  $ I\times \partial I\subset S$.    Let $\delta(m'_1)$ be the closure of $m'_1\setminus  S$.   Using a minimal length tail we can assume that for some $i,j$, $m'_1\cap Z^i_j$ is a single $\mI(S)$ fiber.  Construct an immersed compressing disc $D_1$ for $S$ by taking two parallel copies of $m'_1$ and connecting them by a band that wraps once around $Z^i_j$.  See \cite{Li2} Figure 5.3a.  Since by Lemma \ref{holonomy}  $\tau_i, \tau_{i+1}$ are not homotopic it follows that  $\partial D_1$ is essential in $S$.    If $S$ is 1-sided, then $D_1$ is an immersed essential compressing disc with embedded interior whose boundary has at most one  point of self intersection, hence by Lemma \ref{1-sided} we obtain a contradiction.     From now on we will assume that $S$ is 2-sided, $D_1$ is embedded and $D_1\subset H_1$. \vskip 8pt 


If $D$ contains two monogons $m_1, m_2$ with $m_i\in H_i$, then construct a weak reduction as follows.  Transversely orient $D$ 
so that $P$ lies to the +-side.  Let $D_1$ be as above and let $m'_2$ consist of  $m_2$ together a  tail that spirals around $\piinv
(\partial D)$ enough times to have its end  lie well above $D_1$, i.e. if $\delta m'_2=\{d_u, d_{u+1}\}$, then $u>\max \{r|d_r\cap 
D_1\neq\emptyset\}+2|ad|$.  Create  the disc $D_2\subset H_2$ as follows.  Start with  two parallel copies of $m'_2$ that lie to 
negative side of $D$  (and hence disjoint from $\piinv(P)$), then build a compressing disc using them and the annulus $Z^2_u$.
\vskip 8pt

If all the monogons of $D $ lie in $H_1$, proceed as follows.  Strong irreducibility implies that $A_2\cap D_1\neq\emptyset$ and hence $A_2\cap D\neq\emptyset$.    Consider an arc $\eta$ of $A_2\cap D$ that is outermost in $A_2$ bounding the half-disc $F_\eta\subset A_2$.  If $\eta\cap S$ lies in the same component of  $S\cap \inte(D)$, then  $F_\eta$ together with a disc lying in $D$, slightly isotoped, is a (possibly trivial) compression disc $D'$ for $H_2$ disjoint from $D_1$.  If $D'$  is inessential, then $A_2$ can be isotoped to eliminate the arc $\eta$ of intersection.    Next suppose that $D$ contains exactly two monogons and that $\eta$ has endpoints in each.  In that case there exists a disc $D_\eta\subset D$ such that $\inte(D_\eta)\cap B'_S=\emptyset, \partial D_\eta\subset B'_S\cup \eta$ and $ D_\eta$ has a single cusp.  Now consider the monogon $m_2=F_\eta\cup D_\eta$.  If $F_\eta$ lies to the negative side of $D$, then use $m_2$ to construct a second compression disc as in the previous paragraph.  If $F_\eta$ lies to the +-side of $D$, then swap the roles of $m_2$ and $m_1$ to construct a weak reduction.  I.e. first use $m_2$ and a tail in $\piinv(\partial D$) to construct the first compressing disc $D_2$ and then construct the second compressing disc $D_1$ using two copies of $m_1$ pushed slightly to the negative side of $D$. If $\eta$ does not intersect all the monogons of $D$, then via an isotopy supported in a small neighborhood of $F_\eta$ isotope $B'_S$  so that the resulting $A_2$ intersects $D$ in one fewer component.  Since our new $B'_S$ bounds at least one monogon in $H_1$ and satisfies the conclusion of Lemma \ref{no circles}, Theorem \ref{technical} follows by induction.\qed\vskip 10pt

\newpage

\section{Towards a completely combinatorial proof}

This paper assumes that we to start out with a special triangulation on the hyperbolic 3-manifold N.  A proof of the following conjecture \ref{conjecture} is what is  needed to give a purely combinatorial proof starting with any triangulation.  

\begin{definition}  We say that the branched surface $B\subset M$ is \emph{quasi-hyperbolic} if it does not carry any sphere or torus, but fully carries a surface.  \end{definition}


\begin{example}  As demonstrated in the first paragraph of the proof of Corollary \ref{no annuli} a fully carrying $\eta$-negatively curved branched surface is quasi-hyperbolic.   \end{example}

\begin{conjecture} \label{conjecture} Let $B$ be a branched surface in the compact atoroidal 3-manifold N.  Then there exist finitely many effectively constructible  quasi-hyperbolic branched surfaces $B_1, \cdots, B_n$ such that each $B_i$ is the result of passing to a subbranched surface of some regular splitting of $B$ and every strongly irreducible or incompressible surface carried by $B$ is carried by some $B_i$.\end{conjecture}

\begin{remarks}  i)  A proof of a non effective version of this conjecture is given in the proof of Theorem 1.3 \cite{Li1}.  There the passage from $B$ to $B_1, \cdots, B_n$ is obtained via
a compactness argument.  

ii) The case when $N$ is non Haken and hyperbolic and the surface is a strongly irreducible Heegaard surface, is what is needed for this paper.
\end{remarks}


\begin{definition}  Let $X$ be a Riemannian manifold and $x\in X$.  Define the \emph{hyperbolic injectivity radius} $\hyp_x(X)$ to be  
$$\inf\{r|\textrm{ in}_*(\pi_1(B(x,r))) \textrm{ is infinite, non abelian and Gromov-hyperbolic}\}$$ 
where in$:B(x,r)\to X$ is the inclusion map or $\infty$ if the infimum does not exist.  
Define the  \emph{hyperbolic injectivity radius} of $X$ to be  $\sup\{{x\in X}|\hyp_x(X)\}.$\end{definition}


\begin{example}  As demonstrated in the first paragraph of the proof of Corollary \ref{no annuli} a fully carrying $\eta$-negatively curved branched surface is quasi-hyperbolic.   \end{example}

\begin{theorem}  \label{injectivity bound}Let $B\subset M$ be a quasi-hyperbolic branched surface, then there exists a computable $h(B)>0$ such that if $S$ is a closed surface carried by $B$, and $x\in S$, then $\hyp_x(S)\le h(B)$.  Furthermore any disc or annulus $A$ carried by $B$ with $\partial A\subset b(B)$ (the branch locus) satisfies a computable linear isoperimetric inequality.  \end{theorem}

\begin{remark} The condition that $B\subset M$ is equivalent to saying that $B$ comes equipped with an $I$-bundle structure.\end{remark}

\begin{proof}  We will assume that the branch locus $b(B)$ cuts $B$ into discs, for in general one can subdivide the sectors into discs and readily modify the argument below to deal with this.  Second, we will assume without loss that the length of each edge of $b(B)$ is one. the diameter of each sector $\sigma$ is bounded above by $\length(\sigma)$ and the area of each $\sigma$ is one.  Finally we will assume that if $v$ is a vertex of $\sigma$, then the interior angle of $v$ is either $\pi/2$ or $\pi$.  Note that at a vertex $w$ of $B$, six sectors locally have $w$ as a vertex.  In the natural way, the angle at $w$ of four of these sectors is equal to $\pi/2$ and the angle at the other two is equal to $\pi$. 
    
 It suffices to prove the second statement.  Indeed by Theorem \ref{plante} there is an effectively constructible polynomial that bounds the growth of any surface with base point carried by $B$.  On the other hand it follows from ideas of Gromov \cite{Gr} that the growth rate of discs and annuli with a base point that satisfies a uniform linear isoperimetric inequality is computably exponential, at least up to the distance from the basepoint to the boundary.   
 
Here is an outline of the well known non effective proof that discs satisfy a linear isoperimetric inequality, e.g. see \cite{MO}.    If not by Gromov \cite{Gr} there exist compact discs $D_1, D_2, \cdots$ carried by $B$ with $\partial D_i\subset b(B)$ such that $\ \lim_{i\to \infty} \length(\partial D_i)/\area(D_i)\to 0$.  After passing to subsequence and suitably rescaling, $\{D_i\}\to \mL$ a measured lamination with $\chi(\mL)=0$.  Next approximate $\mL$ by one with compact leaves having the same property.  A leaf of that lamination is either a sphere or a torus.

We now prove the second conclusion of the theorem.  Let $\sigma_1, \cdots, \sigma_n$ denote the sectors of $B$.  Recall that the projective solution space $\mP(B)$ of $B$ (first defined by Haken \cite{Ha} in a slightly different context) is the compact convex subset $\mP$ of the first octant of $\BR^n$ which is  the zero set of finitely many linear equations intersected with the simplex $J$ defined by $t_1x_1+\cdots t+ t_n x_n where \sum t_i=1$.  Here $x_i$ is associated to $\sigma_i$ and to each edge of $b(B)$ is associated  an equation of the form $x_i+x_j-x_k=0$. Note that there is a 1-1 correspondence between normal isotopy classes of closed surfaces carried by $B$ and non negative integral solutions and there is a 1-1 correspondence between normal isotopy classes of weight 1 measured laminations carried by $B$ and points $x\in P$.  $P$ is  a finite sided polygon and each vertex $v$ of $P$ has rational coordinates.  Minimally clearing denominators gives an integral solution corresponding to a closed surface $S_v$ carried by $B$, called a \emph{vertex surface}.  See \cite {O} for a more detailed discussion.  By replacing all these equations by inequalities $ |x_i+x_j-x_k|\le \epsilon$  and restricting to $J$ we obtain $P_\epsilon$ where $P_\epsilon\to P$ as $\epsilon \to 0$.

		
Each $\sigma_i$ has a cell structure induced from $B$.  Let $e^i$ be the set of edges of $\sigma_i$ and  $v^i_{\pi/2}$ (resp. $v^i_{\pi}$) be the set of vertices in $\partial \sigma_i$ with angle $\pi/2$ (resp. $\pi$). Define $X(\sigma_i)=1-|e^i|/2+|v^i_{\pi/2}|/4 +|v^i_{\pi}|/2$ and extend  linearly to $\BR^n$.   Note that if $S$ is a closed surface carried by $B$ with normal coordinates $(a_1, \cdots, a_n)$, then $\chi(S)=X(x)$ where  $x=a_1x_1 + \cdots| + a_nx_n$.  Since $B$ is quasihyperbolic and $X$ is linear it takes a minimum $-C_0<0$ at a vertex.  Since all equations are linear there are computable values $\epsilon_1, C_1>0$ such that for each $x\in P_\epsilon, X(x)<-C_1$.

To complete the proof, first observe that if $A$ is a disc or annulus carried by $B$ with $\partial A\subset b(B)$, then $A=tx$ for some $x\in J$ where 
$t=area(A)>1$.  If $x\notin P_{\epsilon_1}$, then there exists an edge $e$ such that $\partial A$ goes over $e$ at least $t\epsilon_1$ times and hence $
\area(A)<\epsilon_1 \length(\partial A)$. 

Now assume that $x\in P_{\epsilon_1}$.   Consider the induced cellulation on $A$.  Let $e$ be the edges of $\partial A$ and for $i\in \BN$ let $v_i$ be the vertices of $\partial A$ whose interior angle (i.e. the angle on the $A$-side) is equal  to $i(\pi/2)$.  Then 

\begin{align*} 
C_1\area(A)&=0-(-\area(A) C_1)\\
&\le \chi(A)-X(x)\\
&=-|e|/2+\sum_{i\in \BN} (1-i/4)|v_i|\le 3|v_1|/4 +|v_2|/2 + |v_3|/4\\
&\le \length(\partial A)
\end{align*}\end{proof}

		 
\appendix

\section{Mean convex MCF foliations} 


\begin{definition} 
A \emph{mean convex foliation} on a Riemannian $n$-manifold with boundary is a smooth codimension-1 foliation, possibly with singularities of standard type, such that each leaf is closed and mean convex.  
\end{definition}

In a $3$-manifold a foliation with singularities of "standard type" means that almost all leaves are completely smooth (i.e., without any singularities). In particular, any connected subset of the singular set is completely contained in a leaf.  Moreover, the entire singular set is contained in finitely many (compact) embedded Lipschitz curves with cylinder singularities together with a countable set of spherical singularities. The higher dimensions is a direct generalization of this; cf. \cite{CM8}.  

Let $M^{n+1}$ be an (orientable) $(n+1)$-manifold, $\Sigma^n\subset M$ a closed embedded minimal hypersurface with index at least one and suppose that $\Sigma$ bounds a domain $\Omega\subset M$.  Let $L$ be the second variation operator of $\Sigma$, see for instance (1.147) in \cite{CM6}, so that 
\begin{align}
L=\Delta_{\Sigma}+|A|^2+\Ric_M(\nn,\nn)\, .
\end{align}
Here $\nn$ is the unit (inward) normal to $\Sigma$.  Since $\Sigma$ has index at least one the lowest eigenvalue $\lambda$ of $L$ is negative and if $\phi$ is an eigenfunction for $L$ with eigenvalue $\lambda$, then $|\phi |>0$ (see, for instance, lemma 1.35 in \cite{CM6}) and so after possibly replacing $\phi$ by $-\phi$ we may assume $\phi>0$. By the second variation formula if $\Sigma_s=F(x,s)$ is a variation of $\Sigma=\Sigma_0$ with
\begin{align}
F_s&\perp \Sigma_s\, ,\\
F_s(\cdot,0)&=\phi\,\nn_{\Sigma}\, ,
\end{align}
then, see, for instance, theorem 3.2 in \cite{HP}, 
\begin{align}
\frac{d}{ds}_{s=0}H_{\Sigma_s}=-L\,\phi=\lambda\,\phi<0\, .
\end{align}
Here $H$ is the mean curvature scalar, that is, $\text{div}_{\Sigma_s}(\nn_{\Sigma_s})$.  It follows from this that for $s_0>0$ sufficiently small the hypersurface $\Sigma_{s_0}$ is contained in $\Omega$ and is mean convex with respect to the outward normal $-\nn_{\Sigma_{s_0}}$. 

We can now apply the results of \cite{CM8} and flow from $\Sigma_{s_0}$ to obtain the following result.


\begin{theorem} \label{mcf neighborhood} Let $N$ be a Riemannian $n$-manifold, $\Sigma$ a codimension-1 mean convex submanifold 
that is either non minimal or minimal of index $\ge 1$.  Then there exists a smooth codimension-0 submanifold $M$ bounded by $\Sigma$ 
and  a (possibly trivial) stable hypersurface $\Gamma\subset \Omega$ that has a mean convex foliation.    If $\Sigma$ is minimal, then 
such a $\Gamma$ can be chosen to either side of $\Sigma$. \end{theorem}

\begin{proof}  Between $\Sigma$ and $\Sigma_{s_0}$ the mean convex hypersurfaces $\Sigma_s$ foliates and between $\Sigma_{s_0}$ and $\Gamma$ MCF gives a mean convex foliation.  \end{proof}


Using results announced in \cite{PR} (see also \cite{Ru} and compare with \cite{Ke}) we have the following result.  

\begin{theorem} \label {mcf foliation} Any closed orientable bumpy\footnote{Recall that bumpy means that there are no closed immersed minimal surfaces with a nontrivial Jacobi field.  By a result of Brian White such metrics are generic.} Riemannian 3-manifold M with a strongly irreducible Heegaard splitting, supports a mean convex foliation.\end{theorem}

\begin{proof}  By \cite{PR} $M$ either has a minimal strongly irreducible Heegaard surface $\Sigma$ or a stable 1-sided surface also called $\Sigma$ such that the closed complement of $\Sigma$ is either two or one handlebodies, where one occurs exactly in the 1-sided case.  If $\Sigma$ is of index $\ge 1$, then applying the proof of Theorem \ref{mcf neighborhood} to both sides of $\Sigma$ give a compact submanifold $M_0$ whose boundary consists of $\le 2$ components each of which is a stable minimal surface that bounds a handlebody disjoint from $M_0$.  

We have thus reduced to the case that $M$ is a Riemannian handlebody (possibly a 3-ball) with stable boundary.  In what follows we also need to consider the case that $M$ is topologically a product $S\times I$, with stable boundary.  In the former case the standard Heegaard splitting is strongly irreducible, so again the relative version of \cite{PR} described in the sketch of Theorem 1.8 \cite{Ru} shows that $\inte(M)$ supports a minimal Heegaard surface $\Sigma_1$ isotopic to $\partial M$.  In the latter case the same argument shows that $\inte(M)$ supports a minimal surface $\Sigma_1$ isotopic to $S\times 1$.  In either case, if $\Sigma_1$ is index $\ge 1$, then $\Sigma_1$ lies interior to a manifold $M_1$ with a mean convex foliation such that each component of $\partial M_1$ is stable.  Furthermore, each component of the closed complement of $M_1$ in $M$ is either a handlebody or a product.  Since for any $g$, there are only finitely many minimal surfaces in $M$ of genus $\le g$, the result follows.\end{proof}


The proof of Theorem \ref{mcf foliation} shows the following.

\begin{corollary}  
Let $M$ be a compact bumpy Riemannian 3-manifold with mean convex boundary.  If $\{S_i\}$ is a maximal collection of pairwise disjoint minimal surfaces such that each component of the closed complement has a strongly irreducible Heegaard splitting, then $\{S_i\}$ extends to a mean convex foliation on $M$.
\end{corollary}

\begin{remark}  
The hypothesis holds if the closed complement of $\{S_i\}$ in $M$ consists of handlebodies and products.
\end{remark}


\begin{thebibliography}{CEG} 

\bibitem[Br]{Br} W. Breslin, \emph{Curvature bounds for surfaces in hyperbolic 3-manifolds}, Canad. J. Math. \textbf{62} (2010), 994--1010

\bibitem[CaGo]{CaGo} A. Casson \& C. Gordon, \emph{Reducing Heegaard splittings}, Topol. Appl. \textbf{27} (1987), 275--283.



\bibitem[CgT]{CgT}
S.Y. Cheng \& J. Tysk, \emph{An index characterization of the catenoid and index bounds for minimal surfaces in $\RR^4$}, Pacific J. Math. 134 (1988), no. 2, 251--260.

\bibitem[CGK]{CGK} T. H. Colding, David Gabai, Daniel Ketover, \emph{On the Classification of Heegaard Splittings}, preprint.


\bibitem[CM1]{CM1}
T.H. Colding \& W.P. Minicozzi, II, 
\emph{Embedded minimal surfaces without area bounds in 3-manifolds}, 
Geometry and topology: 
Aarhus (1998), 107--120, 
Contemp. Math., 258, AMS, 
Providence, RI, 2000.


\bibitem[CM2]{CM2}
\bysame,
\emph{Estimates for parametric elliptic integrands}, 
Int. Math. Res. Not. (2002), Issue 6,  291--297.

\bibitem[CM3]{CM3}
\bysame, 
\emph{The space of embedded minimal surfaces of fixed genus in a 3-manifold. I. 
Estimates off the axis for disks},  
Annals of Math. (2) 160 (2004), 
no. 1, 27--68.

\bibitem[CM4]{CM4}
\bysame, 
\emph{The space of embedded minimal surfaces of fixed genus in a
3-manifold V; fixed genus}, Annals of Math., vol. 181 (2015), no. 1, 1--153.

\bibitem[CM5]{CM5}
\bysame, 
\emph{The Calabi-Yau conjectures for embedded surfaces}, 
Annals of Math. (2) 167 (2008), no. 1, 211--243.

\bibitem[CM6]{CM6}
\bysame, 
\emph{A course in minimal surfaces},  
Graduate Studies in Mathematics, 
121. American Mathematical Society, 
Providence, RI, 2011.

\bibitem[CM7]{CM7}
\bysame, 
The space of embedded minimal surfaces of fixed genus in a
$3$-manifold III; Planar domains, Annals of Math., 160 (2004) 523--572.

\bibitem[CM8]{CM8}
\bysame, 
The singular set of mean curvature flow with generic singularities, 
http://arxiv.org/1405.5187.

\bibitem[Fc]{Fc}
D. Fischer-Colbrie, \emph{On complete minimal surfaces with finite Morse index in three-manifolds}, Invent. Math. 82 (1985), no. 1, 121--132.

\bibitem[Gr]{Gr} M. Gromov, \emph{Hyperbolic groups}, MSRI Pub. \textbf{8} (1987), 75--263.

\bibitem[Ha]{Ha} W. Haken, \emph{Theorie der Normalflaechen}, Acta Math. \textbf{105} (1961), 245--375.

\bibitem[HP]{HP} G. Huisken and A. Polden, \emph{Geometric evolution equations for hypersurfaces}, in Calculus of Variations and Geometric Evolution Problems (Cetraro, 1996), Lecture Notes in Math. 1713, Springer-Verlag, New York, 1999, pp. 45--84.

\bibitem[JO]{JO} W. Jaco \& U. Oertel, \emph{An algorithm to decide if a 3-manifold is a Haken manifold}, Topology \textbf{23} (1984), 195--209.

\bibitem[Jo]{Jo} K. Johannson \emph{Topology and combinatorics of 3-manifolds}, Lecture Notes in Mathematics, \textbf{1599} (1995), Springer-Verlag, Berlin.

\bibitem[Ke]{Ke} D. Ketover, \emph{Degeneration of Min-Max Sequences in 3-manifolds}, preprint.

\bibitem[Li1]{Li1} T. Li, \emph{Heegaard surfaces and measured laminations I:  The Waldhausen conjecture}, Invent. Math. \textbf{167} (2007), 135-177.

\bibitem[Li2]{Li2} 
\bysame, \emph{Heegaard surfaces and measured laminations, II: non-Haken 3-manifolds}, J. AMS \textbf{19} (2006), 625-657.

\bibitem[Li3]{Li3} 
\bysame, \emph{An algorithm to determine the Heegaard genus of a 3-manifold}, Geometry \& Topology \textbf{15} (2011), 1029-1106. 

\bibitem[LR]{LR}
F. Lopez \& A. Ros, \emph{Complete minimal surfaces with index one and stable constant mean curvature surfaces}, Comment. Math. Helv. 64 (1989), no. 1, 34--43.

\bibitem[Ma]{Ma} J. Manning, \emph{Algorithmic detection and description of
hyperbolic structures on closed 3-manifolds
with solvable word problem}, Geometry \& Topology, Volume \textbf{6} (2002) 1--26.

\bibitem[MO]{MO} L. Mosher \& U. Oertel \emph{Two-dimensional measured laminations of positive Euler characteristic},
Q. J. Math. \textbf{52} (2001), 195--216. 

\bibitem[O]{O} U. Oertel, \emph{measured laminations in 3-manifolds}, T. AMS \textbf{305} (1988), 531--573.

\bibitem[Pl]{Pl} J. F. Plante, \emph{Foliations with measure preserving holonomy}, Annals of Math. \textbf{102} (1975), 327--361.

\bibitem[PR]{PR}  J. Pitts \& J. Rubinstein, \emph{Applications of minimax to minimal surfaces and the topology of 3-manifolds}, Proc. Center Math. Applic., Australian National University \textbf{12}, (1987), 137--170.

\bibitem[Ru]{Ru} J. Rubinstein, \emph{Minimal surfaces in geometric 3-manifolds}. Global theory of minimal surfaces, 725�746, Clay Math. Proc., \textbf{2}, (2005), 725--746.

\bibitem[Sc]{Sc} M. Scharlemann, \emph{Local detection of strongly irreducible Heegaard splittings} Topol. App. \textbf{90} (1998), 135--147. 

\bibitem[S1]{S1}
R.M. Schoen, 
\emph{Uniqueness, symmetry, and embeddedness of minimal surfaces},  
J. Differential Geom. 18 (1983), no. 4,  791--809.

\bibitem[S2]{S2}
\bysame, 
\emph{Estimates for stable minimal surfaces in three-dimensional manifolds. Seminar on minimal submanifolds}, 
111--126, 
Ann. of Math. Stud., 103, 
Princeton Univ. Press, 
Princeton, NJ, 1983.


\bibitem[Wa]{Wa} F. Waldhausen, \emph{Some problems on 3-manifolds}, Proc. Symp. Pure Math. \textbf{32} (1978), 313--322.

\end{thebibliography}
\enddocument